\documentclass[runningheads, orivec]{llncs}
\usepackage[T1]{fontenc}
\usepackage{graphicx}
\usepackage{hyperref}
\usepackage{color}

\usepackage{amsmath}

\usepackage{amssymb}
\usepackage{textgreek}

\newcommand{\dom}{\textbf{Dom}}
\newcommand{\tuple}{\vec}

\newcommand {\Rng}{\texttt{Rng}}

\newcommand {\A}{\mathfrak A}
\newcommand {\B}{\mathfrak B}
\newcommand {\M}{\mathfrak M}

\newcommand{\FO}{\mathbf{FO}}
\newcommand{\ESO}{\mathbf{ESO}}
\newcommand{\D}{\mathbf D}
\newcommand{\E}{\mathbf E}

\newcommand{\DD}{\mathcal D}
\newcommand{\EE}{\mathcal E}
\newcommand{\NE}{\textbf{NE}}

\spnewtheorem{thm}{Theorem}{\bfseries}{\itshape}
\counterwithin{thm}{section}
\spnewtheorem{propo}[thm]{Proposition}{\bfseries}{\itshape}
\spnewtheorem{coro}[thm]{Corollary}{\bfseries}{\itshape}
\spnewtheorem{lem}[thm]{Lemma}{\bfseries}{\itshape}
\spnewtheorem{defin}[thm]{Definition}{\bfseries}{\itshape}

\begin{document}
\title{Strongly First Order Disjunctive Embedded Dependencies in Team Semantics}
\titlerunning{Strongly First Order DEDs in Team Semantics}
%
%
\author{Pietro Galliani\inst{1}\orcidID{0000-0003-2544-5332} }
\authorrunning{P. Galliani}
%
\institute{Università degli Studi dell'Insubria, Via J.H. Dunant, 3
21100 Varese, Italy
\email{pietro.galliani@uninsubria.it}}
\maketitle              
\begin{abstract}
First Order Team Semantics is a generalization of Tarskian Semantics in which formulas are satisfied with respect to sets of assignments. In Team Semantics, it is possible to extend First Order Logic via new types of atoms that describe dependencies between variables; some of these extensions are strictly more expressive than First Order Logic, while others are reducible to it. 

Many of the atoms studied in Team Semantics are inspired by Database Theory and belong in particular to the class of Disjunctive Embedded Dependencies, a very general family of dependencies that contains most of the dependencies of practical interest in the study of databases. 

In this work, I provide a characterization for the (domain-independent) Disjunctive Embedded Dependencies that fail to increase the expressive power of First-Order Team Semantics when added to it. 
\keywords{Team Semantics \and Non-Classical Logic \and Higher-Order Logic}
\end{abstract}

\section{Introduction}
First Order Team Semantics \cite{hodges97} is a generalization of Tarski's semantics in which formulas are satisfied by sets of assignments, called \emph{Teams}, rather than by single assignments. This semantics is reducible to Tarskian Semantics and associates the same truth conditions to first order sentences; but, beginning with V\"a\"an\"anen's work on functional dependence atoms \cite{vaananen07}, Team Semantics was recognized as a natural framework for extending First Order Logic via additional atoms that specify dependencies between the possible values of variables. 

Some of the specific logics thus generated, first and foremost V\"a\"an\"anen's Dependence Logic, have been examined in some depth by now; however, many fundamental questions about the collective properties of this family of extensions of First Order Logic remain open. 

In particular, some atoms, like for instance functional dependence atoms, when added to First Order Logic ($\FO$) bring the expressive power of the resulting formalism all the way up to that of Existential Second Order Logic ($\ESO$) \cite{vaananen07}; others, like inclusion atoms, yield logics whose expressive power is intermediate between $\FO$ and $\ESO$ \cite{gallhella13}; and others yet, like constancy and nonemptiness atoms, fail to increase the expressive power of $\FO$ at all \cite{galliani2015upwards}. 

A general, effective classification of which dependencies (or combinations thereof) belong to these three categories would go a long way to clarify the possibilities of Team Semantics and put some order in the family of the extensions of First Order Logic that can be generated through it. Some characterizations are known for the dependencies that fail to increase the expressive power of $\FO$ (that is to say, that are ``strongly first order'') \textbf{and}  satisfy certain \emph{closure conditions}; however, no such result is yet known for a family of dependencies that is broad enough to include all (or at least most) the dependencies studied so far in Team Semantics and that is natural enough to be of independent interest.

The main contribution of the present work is a characterization of the strongly first order dependencies in the class of (domain-independent) \emph{Disjunctive Embedded Dependencies} (DEDs), which is well known from Database Theory \cite{kanellakis1990elements} and indeed contains most dependencies studied so far in Team Semantics. 

As a consequence of this characterization, we will also see that a family of domain-independent DEDs, if added collectively to $\FO$, increases its expressive power if and only if at least one DED in this family does so individually: in other words, if $\D_1 \ldots \D_n$ are strongly first order domain-independent DEDs then the logic obtained by adding them all to $\FO$ is still equiexpressive to $\FO$. 
\section{Preliminaries}
\subsection*{Notation}
We write $\mathfrak A$, $\mathfrak B$, $\mathfrak C$, \ldots, $\M$ to indicate first order models, and $A$, $B$, $C$, $\ldots M$ for their repective domains. Given a relation symbol $R$, we  write $R^\M$ for the interpretation of $R$ in the model $\M$; and given a constant symbol $c$, we likewise write $c^\M$ for its interpretation in $\M$. Where no ambiguity is possible we identify symbols with the relations and elements that they represent, e.g. we write $\M = (M, R, a)$ for the first order model with domain $M$ such that $R^\M = R \subseteq M^k$ and $a^\M = a \in M$. Given two models $\mathfrak A$ and $\mathfrak B$, we write $\mathfrak A \preceq \mathfrak B$ to say that $\mathfrak B$ is an elementary extension of $\mathfrak A$, and $\mathfrak A \equiv \mathfrak B$ if $\mathfrak A$ and $\mathfrak B$ are elementarily equivalent. 

Given an assignment $s$ from a set of variables $V$ to a domain $A$, an element $a \in A$ and a variable $v \in \text{Var}$, we write $s[a/v]$ for the result of assigning $a$ to $v$ in $s$. If $\tuple v = v_1 \ldots v_k$ is a tuple of variables, we will write $s(\tuple v)$ for the tuple of elements $s(v_1) \ldots s(v_k)$. If $\tuple a = a_1 \ldots a_k$ is a tuple of elements, we write $\Rng(\tuple a) = \bigcup_{i=1}^k a_i$ for the set of all elements occurring in $\tuple a$; and we say that a tuple of elements $\tuple b^{(1)}$ is \emph{disjoint} from a tuple $\tuple b^{(2)}$ \emph{except on} $\tuple a$ if $\Rng(\tuple b^{(1)}) \cap \Rng(\tuple b^{(2)}) \subseteq \Rng(\tuple a)$. Likewise, we say that a tuple $\tuple b$ is disjoint from a set $A$ except on a tuple $\tuple a$ if $\Rng(\tuple b) \cap A \subseteq \Rng(\tuple a)$. We say that a tuple $\tuple a$ \emph{lists} (or \emph{is a list of}) a finite set $A$ if $\tuple a$ has no repetitions and $\Rng(\tuple a)=A$.

Given a tuple of elements $\tuple a = a_1 \ldots a_k \in A^k$, its \emph{identity type} is $\tau(x_1 \ldots x_k) := \bigwedge_{a_i = a_j} (x_i = x_j) \wedge \bigwedge_{a_i \not = a_j} (x_i \not = x_j)$.

Given two extensions $L_1$ and $L_2$ of First Order Logic, we write $L_1 \leq L_2$ if every sentence of $L_1$ is equivalent to some sentence of $L_2$; $L_1 \equiv L_2$ if $L_1 \leq L_2$ and $L_2 \leq L_1$; and $L_1 < L_2$ if $L_1 \leq L_2$ but $L_2 \not \leq L_1$. $\FO$ represents First Order Logic itself, and $\ESO$ represents Existential Second Order Logic.
\subsection*{Team Semantics}
We will now recall the basic definitions and properties of Team Semantics.\footnote{We only consider the more common ``\emph{lax}'' version of this semantics, corresponding to a non-deterministic form of Game-Theoretic Semantics. There also exists a ``\emph{strict}'' variant which corresponds in a similar way to a deterministic Game-Theoretic Semantics, but since for that variant the property of \emph{locality} fails to hold in general (that is to say, the satisfaction conditions of a formula may depend on variables that are not free in it: see \cite{galliani12} for a more detailed discussion) it is usually preferred to work with the lax form of Team Semantics.}
\begin{defin}[Team]
Let $\A$ be a first order model and let $V \subseteq \text{Var}$ be a finite set of variables. A team $X$ over $\A$ with domain $\dom(X) = V$ is a set of assignments $s: V \rightarrow A$. Given a team $X$ and a tuple of variables $\tuple v$, we write $X(\tuple v)$ for the relation $X(\tuple v) = \{s(\tuple v) : s \in X\}$; and given two teams $X$ and $Y$, we write that $X \equiv_{V} Y$ if $X(\tuple v) = Y(\tuple v)$ for a list $\tuple v$ of $V$. 
\end{defin}

\begin{defin}[Team Semantics]
Let $\phi$ be a first order formula in negation normal form,\footnote{It is common in the study of Team Semantics to require that all expressions are in Negation Normal Form. This is because, in general, there is no obvious interpretation for the negation of a dependency atom as well as because, in the context of Team Semantics, the contradictory negation $\M \models_X \sim \!\phi \Leftrightarrow \M \not \models_X \phi$ would bring the expressive power of most logics all the way up to full Second Order Logic \cite{vaananen07,kontinennu11}.} let $\M$ be a model whose signature contains that of $\phi$, and let $X$ be a team over $\M$ whose domain contains the free variables of $\phi$. Then $\phi$ is satisfied by $X$ in $\M$, and we write $\M \models_X \phi$, if this follows from the following rules:
\begin{description}
\item[TS-lit:] If $\phi$ is a literal, $\M \models_X \phi$ $\Leftrightarrow$ $\forall s \in X$, $M \models_s \phi$ in Tarskian semantics; 
\item[TS-$\vee$:] $\M \models_X \psi_1 \vee \psi_2$ $\Leftrightarrow$ $\exists X_1, X_2$ s.t. $X = X_1 \cup X_2$, $M \models_{X_1} \psi_1$ and $M \models_{X_2} \psi_2$; 
\item[TS-$\wedge$:] $M \models_X \psi_1 \wedge \psi_2$ $\Leftrightarrow$ $\M \models_X \psi_1$ and $M \models_X \psi_2$; 
\item[TS-$\exists$:] $M \models_X \exists v \psi$ $\Leftrightarrow$ $\exists Y$ s.t. $\dom(Y) = \dom(X) \cup \{v\}$, $Y \equiv_{\dom(X) \backslash \{v\}} X$, and $M \models_Y \psi$; 
\item[TS-$\forall$:] $M \models_X \forall v \psi$ $\Leftrightarrow$  $M \models_{X[M/v]} \psi$, for $X[M/v] = \{s[m/v] : s \in X, m \in M\}$. 
\end{description}

If $\phi$ is a sentence, we write that $\M \models \phi$ (in the sense of Team Semantics) if and only if $\M \models_{\{\varepsilon\}} \phi$, where $\varepsilon$ is the empty assignment. 
\label{def:TS}
\end{defin}
When working with First Order Logic proper there is no reason to use Team Semantics rather than the simpler Tarskian Semantics: 
\begin{propo}(\cite{vaananen07}, Corollary 3.32)
Let $\M$ be a model, let $\phi$ be a first order formula in negation normal form over the signature of $\M$, and let $X$ be a team over $\M$. Then $\M \models_X \phi$ if and only if, for all $s \in X$, $\M \models_s \phi$ in the usual Tarskian sense. In particular, if $\phi$ is a first order sentence in negation normal form, $\M \models \phi$ in Team Semantics if and only if $\M \models \phi$ in Tarskian Semantics. 
	\label{propo:flat}
\end{propo}
However, Team Semantics allows one to augment First Order Logic in new ways, for example via new \emph{dependence atoms} such as Functional Dependence Atoms \cite{vaananen07}, Inclusion Atoms \cite{galliani12,gallhella13}, Independence Atoms \cite{gradel13} or Anonymity Atoms \cite{vaananen2023atom}:\footnote{These atoms could be defined so that they can apply to tuples of terms and not only to tuples of variables. Since here we always operate in logics at least as expressive as First Order Logic and we are not restricting existential quantification, for simplicity's sake we limit ourselves to tuples of variables.}
\begin{description}
\item[TS-func:] $\M \models_X =\!\!(\tuple v;\tuple w)$ $\Leftrightarrow$ $\forall s, s' \in X$, if $s(\tuple v) = s'(\tuple v)$ then $s(\tuple w) = s'(\tuple w)$; 
\item[TS-inc:] $M \models_X \tuple v \subseteq \tuple w$ $\Leftrightarrow$ $X(\tuple v) \subseteq X(\tuple w)$; 
\item[TS-ind:] $M \models_X \tuple v ~\bot~ \tuple w$ $\Leftrightarrow$ $X(\tuple v \tuple w) = X(\tuple v) \times X(\tuple w)$; 
\item[TS-anon:] $M \models_X \tuple v ~\mbox{\textUpsilon}~ \tuple w$ $\Leftrightarrow$ $\forall s \in X$ $\exists s' \in X$ s.t. $s(\tuple v) = s'(\tuple v)$ but $s(\tuple w) \not = s'(\tuple w)$. 
\end{description}
The logics obtained by adding them to First Order Logic are called respectively Dependence Logic, Inclusion Logic, Independence Logic and Anonymity Logic. In order to study the family of all logics that are obtainable in such a way, it is convenient to use the following notion of \emph{generalized dependency} \cite{kuusisto13}: 
\begin{defin}[Generalized Dependency]
Let $R$ be a $k$-ary relation symbol and let $\D$ be a class, closed under isomorphisms, of models over the signature $\{R\}$. Then $\D$ is a $k$-ary \emph{generalized dependency} and $\FO(\D)$ is the logic obtained by adding to First Order Logic the atoms $\D \tuple v$ for all $k$-tuples of variables $\tuple v$, with the satisfaction conditions 
\begin{description}
\item[TS-$\D$:] 
$\M \models_X \D \tuple v$ $\Leftrightarrow$ $(M, X(\tuple v)) \in \D$
\end{description}
where $(M, X(\tuple v))$ is the model with domain $M$ in which $R$ is interpreted as $X(\tuple v)$. 

If $\DD$ is a family of generalized dependencies, we write $\FO(\DD)$ for the logic obtained by adding to First Order Logic the atoms corresponding to all $\D \in \DD$. 
\label{def:gendep}
\end{defin}
The atoms mentioned above can all be modeled as families of generalized dependencies: for example, Dependence Logic is obtained by adding to First Order Logic all atoms corresponding to generalized dependencies of the form 
\begin{equation}
    \textbf{Dep}_{n,m} = \{(M, R): (M, R) \models \forall \tuple x \tuple y \tuple z ( R\tuple x \tuple y \wedge  R\tuple x \tuple z \rightarrow \tuple y = \tuple z)\}
    \label{eq:func_gen}
\end{equation}
where $n,m \in \mathbb N$, $\tuple x$ has arity $n$, $\tuple y$ and $\tuple z$ have arity $m$, and as usual $\tuple y = \tuple z$ is a shorthand for $\bigwedge_{i=1}^m y_i = z_i$. Inclusion atoms, independence atoms and anonymity atoms can be likewise represented via first order sentences as in (\ref{eq:func_gen}).  Thus, they are \emph{first order} dependencies in the following sense: 
\begin{defin}[First Order Generalized Dependency]
A generalized dependency $\D$ is first order if there exists a first order sentence $\D(R)$ such that $\D = \{(M, R) : (M, R) \models \D(R)\}$.
If so, with a slight abuse of notation we identify $\D$ with the sentence $\D(R)$. 
\end{defin}

Definition \ref{def:gendep} is, however, arguably \emph{too} general. Indeed, it admits ``dependencies'' like $\textbf E = \{(M, R): |M| \text{ is even}\}$, whose corresponding satisfaction condition does not ask anything of $X(\tuple v)$ but instead says that the \emph{model} has an even number of elements. This is unreasonable: intuitively, a dependence atom $\D \tuple v$ should say something about the possible values of $\tuple v$. This can be formalized as follows:\footnote{This notion first appeared in \cite{kontinen2016decidable}, in which it is called ``Universe Independence''.}
\begin{defin}[Domain-Independent Generalized Dependencies]
A $k$-ary dependency $\D$ is \emph{domain-independent} if, for all domains of discourse $M$, $N$ and all $k$-ary relations $R \subseteq M^k \cap N^k$, $(M, R) \in \D$ if and only if $(N, R) \in \D$.
\end{defin}

It is easy to see that, aside from the pathological ``dependency'' $\textbf E$, all examples of dependencies  that we saw so far are domain-independent. 

What can we say, in general, about the properties of a logic of the form $\FO(\DD)$? 
Clearly, it is always the case that $\FO \leq \FO(\DD)$; and if $\D \in \DD$ is not first order then $\FO < \FO(\DD)$, because the $\FO(\DD)$ sentence 
\begin{equation}
\forall \tuple v(\lnot R \tuple v \vee (R \tuple v \wedge \D \tuple v))
\label{eq:defDinFOD}
\end{equation}
defines the class $\D$, which by hypothesis this is not first-order definable. 

However, the converse is not true. As we saw, functional dependence atoms are first order; and yet, (Functional) Dependence Logic is as expressive as Existential Second Order Logic \cite{vaananen07}. 
Inclusion Logic, Independence Logic and Anonymity Logic are likewise more expressive than First Order Logic; but whereas Independence Logic is also as expressive as Existential Second Order Logic \cite{gradel13}, Anonymity Logic is equivalent to Inclusion Logic (see Propositions 4.6.3. and 4.6.4. of \cite{galliani12c}, in which anonymity is called 'non-dependence') and both are only as expressive as the positive fragment of Greatest Fixpoint Logic \cite{gallhella13}. 

\begin{defin}[Definability of a dependency]
Let $\DD$ be a family of dependencies, and let $\E$ be a $k$-ary dependency. Then $\E$ is \emph{definable} in $\FO(\DD)$ if there exist a tuple of $k$ variables $\tuple v = v_1 \ldots v_k$ and a formula $\phi(\tuple v) \in \FO(\DD)$ over the empty signature such that $\M \models_X \E \tuple v \Leftrightarrow \M \models_X \phi(\tuple v)$
for all $\M$ and $X$. 
\end{defin}
\begin{propo}
Let $\DD$ and $\EE$ be families of generalized dependencies. If every $\E \in \EE$ is definable in $\FO(\DD)$ then $\FO(\EE) \leq \FO(\DD) \equiv \FO(\DD, \EE)$. 
\label{propo:def_expr}
\end{propo}
\begin{proof}
Suppose that for every $\E \in \EE$ there exists a formula $\phi_{\E}(\tuple v) \in \FO(\DD)$ over the empty signature that is equivalent to $\E \tuple v$. Then, renaming variables as needed, we see that every occurrence of $\E \tuple w$ for every $\E \in \EE$ and every tuple of variables $\tuple w$ is equivalent to some formula $\phi_{\E}(\tuple w) \in \FO(\DD)$. Thus, every sentence of $\FO(\EE)$ or of $\FO(\DD, \EE)$ is equivalent to some sentence of $\FO(\DD)$; and of course, every sentence of $\FO(\DD)$ is also a sentence of $\FO(\DD, \EE)$. \qed
\end{proof}

It is sometimes useful to ``restrict'' a team to the assignments that satisfy individually some first order formula, like we did in (\ref{eq:defDinFOD}): 
\begin{defin}[$\theta \hookrightarrow \phi$]
Let $\theta$ be a first order formula and let $\phi$ be a $\FO(\DD)$ formula for some collection $\DD$ of generalized dependencies. Then we write $\theta \hookrightarrow \phi$ for the $\FO(\DD)$ formula $\theta' \vee (\theta \wedge \phi)$, 
where $\theta'$ is the first order formula in Negation Normal Form that is equivalent to $\lnot \theta$. 
\end{defin}
\begin{propo}
For all models $\M$, teams $X$, families of generalized dependencies $\DD$, first order formulas $\theta$ and $\FO(\DD)$ formulas $\phi$, $\M \models_X \theta \hookrightarrow \phi \Leftrightarrow \M \models_{X_{|\theta}} \phi$ for $X_{|\theta} = \{s \in X : M \models_s \theta \text{ in Tarski semantics}\}$.
\label{propo:TS-imp}
\end{propo}
\begin{proof}
Suppose that $\M \models_X \theta \hookrightarrow \phi$. Then $X = Y \cup Z$, where $\M \models_Y \theta'$ (for $\theta'$ being equivalent to $\lnot \theta$), $\M \models_Z \theta$ and $\M \models_Z \phi$. Then by Proposition \ref{propo:flat} we have that $M \models_s \lnot \theta$ (in the ordinary Tarskian sense) for all $s \in Y$ and that $\M \models_s \theta$ for all $s \in Z$; therefore, we necessarily have that $Z = X_{|\theta}$ and $Y = X \backslash Z$, and so $\M \models_{X_{|\theta}} \phi$. 

Conversely, suppose that $\M \models_{X_{|\theta}} \phi$. Let $Z = X_{|\theta}$ and $Y = X \backslash Z$: then by Proposition \ref{propo:flat} we have that $\M \models_Y \theta'$ and $\M \models_Z \theta \wedge \phi$, and so $\M \models_X \theta \hookrightarrow \phi$. \qed. 
\end{proof}

Finally we mention an extra connective that can be added to Team Semantics: the \emph{global disjunction} $\sqcup$ such that $\M \models_X \psi_1 \sqcup \psi_2$ if and only if $\M \models_X \psi_1$ or $\M \models_X \psi_2$. $\FO(\DD, \sqcup)$ represents the logic obtained by adding $\sqcup$ to $\FO(\DD)$. 
\subsection*{Strongly First Order Dependencies}
A special case of functional dependency is the ``constancy dependency'' $=\!\!(\emptyset; \tuple w)$, usually written $=\!\!(\tuple w)$, for which $\M \models_X =\!\!(\tuple w)$ $\Leftrightarrow$ $\forall s, s' \in X, s(\tuple w) = s'(\tuple w)$:
\begin{propo}[\cite{galliani12}, \S 3.2] 
Let $=\!\!(\cdot)$ be the family of all constancy dependencies of all arities. Then $\FO(=\!\!(\cdot)) \equiv \FO$.
\end{propo}

Which other dependencies likewise fail to increase the expressive power of $\FO$? In \cite{galliani2015upwards} the class of the \emph{upwards closed dependencies} (i.e., those such that $(M, R) \in \D, R \subseteq S \subseteq M^k \Rightarrow (M, S) \in \D$) was introduced and the following result was shown:
\begin{thm}[\cite{galliani2015upwards}, Theorem 21]
Let $\DD^\uparrow$ be the family of all first order upwards closed dependencies and let $=\!\!(\cdot)$ be the family of all constancy dependencies. Then $\FO(\DD^\uparrow, =\!\!(\cdot)) \equiv \FO$. 
\label{theo:upcl_sfo}
\end{thm}
Thus, first order upwards closed dependencies and constancy dependencies are ``safe'' for $\FO$ in the sense of \cite{galliani2020safe}. The \emph{non-emptiness atoms} $\NE = \{(M, R) : R \not = \emptyset\}$, for which $\M \models_X \NE(\tuple v) \Leftrightarrow |X(\tuple v)| \not = \emptyset$, belong in $\DD^{\uparrow}$ and will be useful in this work. 
\begin{defin}[Strongly First Order Dependencies \cite{galliani2016strongly}]
A dependency $\D$, or a family of dependencies $\DD$,  is strongly first order if $\FO(\D) \equiv \FO$ (respectively $\FO(\DD) \equiv \FO$).  
\end{defin}
\begin{thm}[\cite{galliani2016strongly}, Corollary 8]
Let $\DD^1$ be the family of all unary first order dependencies. Then $\FO(\DD^1) \equiv \FO$.
\end{thm}


Upwards closed dependencies and unary dependencies are somewhat uncommonly encountered when working with Team Semantics. Instead, the class of \emph{downwards closed dependencies} (i.e. those such that $(M, R) \in \D, S \subseteq R \Rightarrow (M, S) \in \D$) is of special importance, because functional dependencies -- the very first ones studied in Team Semantics -- are in it; and in \cite{galliani2019characterizing}, it was shown that a domain-independent\footnote{This result actually uses a weaker, more technical condition than domain-independence called \emph{relativizability}.} downwards closed dependency is strongly first order if and only if it is definable in $\FO(=\!\!(\cdot))$. This entirely answered the question of which dependencies are strongly first order for a fairly general class of dependencies, which however fails to include many dependencies of interest. 

In \cite{galliani2019nonjumping}, a similar approach was adapted for characterizing strongly first order dependencies that satisfy a more general -- if somewhat technical - closure property. That paper also contains the following result, which we will need: 
\begin{thm}[\cite{galliani2019nonjumping}, Proposition 14]
If $\FO(\DD) \equiv \FO$ then $\FO(\DD, \sqcup) \equiv \FO$.
\label{thm:sqcup_safe}
\end{thm}

In \cite{galliani2022strongly}, it was then shown that if a domain-independent $\D$ is \emph{union-closed}, in the sense that $(M, R_i) \in \D ~\forall i \in I \Rightarrow (M, \bigcup_i R_i) \in \D$, then it is strongly first-order if and only if it is definable in $\FO(=\!\!(\cdot), \NE, \sqcup)$. Examples of union-closed dependencies are the inclusion and anonymity dependencies mentioned before. 

Finally, in \cite{galliani24} a characterization was found for the domain-independent dependencies that are \emph{doubly strongly first order} in the sense that $\FO(\D, \sim\!\D) \equiv \FO$, where $\sim\!\D = \{(M, R) : (M, R) \not \in \D\}$. This can be seen as studying the safety of $\D$ with respect to a richer base language,  in which it is also possible to deny dependencies. 
%
\subsection*{Disjunctive Embedded Dependencies}
One of the concerns of Database Theory is the specification and analysis of dependencies between entries of relational databases \cite{abiteboul1995foundations,Deutsch2009,deutsch2001optimization}. Many such dependencies, like the functional and inclusion dependencies seen above, have been studied in this context, and the following notion of \emph{disjunctive embedded dependency} has been found to suffice for many practically relevant scenarios:
\begin{defin}[Disjunctive Embedded Dependencies]

Let $R$ be a $k$-ary relational symbol. A (unirelational)\footnote{In Database Theory a dependency may involve multiple relations corresponding to different tables; in Team Semantics, however, we only have one relation to work with. 
} \emph{Disjunctive Embedded Dependency} (or DED) over the vocabulary $\{R\}$ is a first order sentence of the form 
\begin{equation}
    \forall \tuple x \left(\phi(\tuple x) \rightarrow \bigvee_i \exists \tuple y^{(i)} \psi_i(\tuple x, \tuple y^{(i)})\right)
    \label{eq:ded}
\end{equation}
where $\phi$ and all $\psi_i$ are conjunctions of relational and identity atoms. 
\end{defin}
Functional dependencies, inclusion dependencies, independence atoms, and nonemptiness atoms are all DEDs, since the sentences $\forall \tuple x \tuple y \tuple z ( (R\tuple x \tuple y \wedge R \tuple x \tuple z) \rightarrow \tuple y = \tuple z)$, $\forall \tuple x \tuple y (R \tuple x \tuple y \rightarrow \exists \tuple z (R \tuple z \tuple x))$, $\forall \tuple x \tuple y \tuple z \tuple w ((R \tuple x \tuple y \wedge R \tuple z \tuple w) \rightarrow R \tuple x \tuple w)$ and $\forall x(x = x \rightarrow \exists \tuple y R \tuple y)$ are of the required form. In fact, since they do not require a disjunction in the consequent they belong to the more restricted class of (non-disjunctive) \emph{embedded dependencies}, which suffices already for many purposes:
\begin{quote}
\emph{Embedded dependencies turn out to be sufficiently
expressive to capture virtually all other classes of
dependencies studied in the literature.} \cite{Deutsch2009}
\end{quote}

Anonymity atoms, however, are not DEDs: for example, they do not satisfy Proposition \ref{propo:DEDhom} below. The even more general class DED$^{\not =}$ allows inequality literals $z\not = w$ inside of (\ref{eq:ded}), and would contain them; but in this work we will limit ourselves to the inequality-free case.

\begin{propo}
Let $\D(R)$ be a DED. Then $\D$ is preserved by unions of chains, in the sense that if $(R_n)_{n \in \mathbb N}$ is a family of relations over $A$ s.t.
\begin{enumerate}
\item $(A, R_n) \in \D$ for all $n \in \mathbb N$; 
\item $R_n \subseteq  R_{n+1}$ for all $n \in \mathbb N$
\end{enumerate}
then $(A, \bigcup_{n \in \mathbb N} R_n) \in \D$. 
\label{propo:DEDuc}
\end{propo}
\begin{proof}
It suffices to observe that DEDs are logically equivalent to $\forall \exists$ sentences. \qed
\end{proof}

%
\begin{propo}
Let $\D(R)$ be a DED, and let $(A, R)$ and $(B, S)$ be such that 
\begin{enumerate}
\item $(A, R)$ is a substructure of $(B, S)$; 
\item There is a homomorphism $\mathfrak h: (B, S) \rightarrow (A, R)$ that is the identity over $A$; 
\item $(B, S) \in \D$. 
\end{enumerate}
Then $(A, R) \in \D$. 
\label{propo:DEDhom}
\end{propo}
\begin{proof}
Suppose that $\D(R)$ is of the form of Equation (\ref{eq:ded}), and take any tuple $\tuple a$ such that $(A, R) \models \phi(\tuple a)$. Since $\phi$ is a conjunction of atoms and $(A, R)$ is a substructure of $(B, S)$, it must be the case that $(B, S) \models \phi(\tuple a)$; and since $(B, S) \models \D(S)$, there exists some $i$ and some tuple $\tuple b$ of elements of $B$ such that $(B, S) \models \psi_i(\tuple a, \tuple b)$. Now, $\mathfrak h(\tuple a) = \tuple a$  and $\psi_i$ is a conjunction of atoms, so for $\tuple c = \mathfrak h(\tuple b)$ we have that $(A, R) \models \psi_i(\tuple a, \tuple c)$; therefore, $(A, R) \models \bigvee_i \exists \tuple y^{(i)} \psi_i(\tuple a, \tuple y^{(i)})$. 
\qed
\end{proof}
One particular consequence of the above result will be useful to us in this work: 
\begin{coro}
Let $\D$ be a domain-independent DED, let $R$ be any relation over some $A$, let $B \supseteq A$, and let $\tuple b_1^{(1)} \ldots \tuple b_n^{(1)},  \tuple b_1^{(2)} \ldots \tuple b_n^{(2)}, \ldots \in B^k \backslash A^k$ be $k$-tuples of elements of $B$ such that, for $\tuple a$ listing the elements of $A \cap \bigcup_{i=1}^n \Rng(\tuple b_i^{(1)})$,
\begin{enumerate}
\item Every $\tuple b_i^{(q)}$ is disjoint from $A$ except on $\tuple a$; 
\item Whenever $q \not = q'$, $\tuple b_1^{(q)} \ldots \tuple b_n^{(q)}$ and $\tuple b_1^{(q')} \ldots \tuple b_n^{(q')}$ are disjoint except on $\tuple a$; 
\item For all $q \in \mathbb N$, the identity type of $\tuple b_1^{(q)} \ldots \tuple b_n^{(q)} \tuple a$ is the same. 
\end{enumerate}
Suppose furthermore that $(B, S) \in \D$ for some $S$ with $R \subseteq S \subseteq R \cup \bigcup_{q \in \mathbb N} \{\tuple b_1^{(q)}, \ldots, \tuple b_n^{(q)}\}$ and that $\{\tuple b_1^{(q)}, \ldots, \tuple b_n^{(q)} : q \in Q\} \subseteq S$ for some nonempty $Q \subseteq \mathbb N$. 

Then $(B, R \cup \{\tuple b_1^{(q)} \ldots \tuple b_n^{(q)} : q \in Q\}) \in \D$ as well. 
\label{coro:DEDhom}
\end{coro}
\begin{proof}
Let $C = A \cup\bigcup_{\tuple r \in S} \Rng(\tuple r)$, $C' = A \cup \bigcup \{\Rng(\tuple b^{(q)}_i): q \in Q, i \in 1 \ldots n\}$ and $S' = R \cup \{\tuple b_1^{(q)}, \ldots, \tuple b_n^{(q)} : q \in Q\}$. We need to prove that $(B, S') \in \D$. 

By the domain independence of $\D$, we know that $(C, S) \in \D$. Note that $S \cap (C')^k = S'$, because all tuples in $S \backslash S'$ must be of the form $\tuple b_i^{(q)}$ for some $q \not \in Q$ (and so must contain some element not in $C'$): therefore, $(C', S')$ is a substructure of $(C, S)$. If we can find a homomorphism $\mathfrak h: (C, S) \rightarrow (C', S')$ that keeps $(C', S')$ fixed, by Proposition \ref{propo:DEDhom} we have that $(C', S') \in \D$; and then by the domain independence of $\D$ we have that $(B, S') \in \D$, as required. 

It remains to define this $\mathfrak h$. Fix an arbitrary $q_0 \in Q$; then, for all $c \in C$, let  
\[
    \mathfrak h(c) = \left\{ \begin{array}{l l}
    c & \text{ if } c \in C'; \\
    (\tuple b^{(q_0)}_i)_j & \text{ if } c \not \in C' \text{ is the } j\text{-th element of } \tuple b^{(q)}_i \text{ for } q \not \in Q, i \in 1 \ldots n.
    \end{array}
    \right.
\]
By construction, $\mathfrak h$ keeps $C'$ fixed; we need to show that it is well-defined and that it is a homomorphism. 

If $c \in C \backslash C'$ then $c$ occurs in at least one $\tuple b_i^{(q)}$ for $q \not \in Q$; but it may occur in more than one. Suppose then that $c = (\tuple b^{(q)}_i)_j =  (\tuple b^{(q')}_{i'})_{j'}$ for $q, q' \not \in Q$. We observe that $q = q'$: indeed, otherwise we would have that $\tuple b^{(q)}_i$ and $\tuple b^{(q')}_{i'}$ intersect only over $\tuple a$, and since $c \not \in C' \supseteq \Rng(\tuple a)$ it cannot be the case that $c$ occurs in $\tuple a$. Then $(\tuple b^{(q)}_i)_j =  (\tuple b^{(q)}_{i'})_{j'}$; and since the identity types of $\tuple b^{(q)}_1 \ldots \tuple b^{(q)}_n \tuple a$ and $\tuple b^{(q_0)}_1 \ldots \tuple b^{(q_0)}_n \tuple a$ are the same, we have that $(\tuple b^{(q_0)}_i)_j = (\tuple b^{(q_0)}_{i'})_{j'}$ and $\mathfrak h$ is indeed well-defined.

Finally, suppose that $(c_1 \ldots c_k) \in S$. If $(c_1 \ldots c_k) \in S'$ then all $c_j$ are in $C'$, and so $(\mathfrak h(c_1) \ldots \mathfrak h(c_k)) = (c_1 \ldots c_k) \in S'$. If instead $(c_1 \ldots c_k) \in S \backslash S'$, it must be the case that $(c_1 \ldots c_k) = \tuple b^{(q)}_i$ for some $q \not \in Q$ and some $i \in 1 \ldots n$. But then have that $(\mathfrak h(c_1) \ldots \mathfrak h(c_k)) = \tuple b^{(q_0)}_i \in S'$. Indeed, for all $j \in 1 \ldots n$, if $c_j \not \in C'$ then $\mathfrak h(c_j) = ( b^{(q_0)}_i)_j$ by definition. If instead $c_j \in C'$ then $\mathfrak h(c_j) = c_j$. But in this case $c_j$ occurs in $\tuple a$, because $\tuple b^{(q)}_i$ and $C'$ intersect only over $\tuple a$; and since the identity types of $\tuple b^{(q)}_1 \ldots \tuple b^{(q)}_n \tuple a$ and $\tuple b^{(q_0)}_1 \ldots \tuple b^{(q_0)}_n \tuple a$ are the same, it follows that $(\tuple b^{(q_0)}_i)_j = (\tuple b^{(q)}_i)_j = c_j$ as required.

This shows that $\mathfrak h$ is a homomorphism and concludes the proof. \qed
\end{proof}
\section{Strongly First Order DEDs}
\subsection*{Some Properties}
The next two propositions are proven as in (\cite{galliani2022strongly}, Lemma 1 and Proposition 3): 
\begin{propo}
Let $\D(R)$ be a strongly first order, $k$-ary DED. Then there cannot exist an infinite chain of $k$-ary relations $R_1 \subseteq S_1 \subseteq R_2 \subseteq S_2 \subseteq \ldots$ over some domain $M$ such that $(M, R_n) \in \D$ and $(M, S_n) \not \in \D$ for all $n \in \mathbb N$. 
\label{propo:noChain}
\end{propo}
\begin{proof}
Suppose that such a chain exists. Then $M$ must be infinite, and without loss of generality we can assume that it contains $\mathbb N$; and since $\bigcup_{n \in \mathbb N} R_n = \bigcup_{n \in \mathbb N} S_n$, by Proposition \ref{propo:DEDuc} it must hold that $(M, \bigcup_{n \in \mathbb N} S_n) \in \D$. 

Now let $\M$ be a first order model with domain $M$, with a unary predicate $N$ such that $N^{\M} = \mathbb N$, a binary predicate $<$ interpreted as the usual order relation over $\mathbb N$, a $(k+1)$-ary relation $R$ such that $R^\M = \{(n, \tuple a) : n \in \mathbb N, \tuple a \in R_n\}$ and a $(k+1)$-ary relation $S$ such that $S^\M = \{(n, \tuple a) : n \in \mathbb N, \tuple a \in S_n\}$. Then for every element $d \in \mathbb N$, the $\FO(\D)$ sentence 
\begin{equation}
    \exists i (i < d \wedge \forall \tuple v(S i \tuple v \hookrightarrow \D \tuple v))
    \label{eq:noD}
\end{equation}
must be false, as it asserts that there exists a family $I \subseteq \{1 \ldots d\}$ of indices such that $(M, \bigcup_{n \in I} S_n) \in \D$ and that is impossible since $\bigcup_{n \in I} S_n = S_{\max(I)}$. Therefore, if $\phi(d)$ is the first order sentence equivalent to (\ref{eq:noD}) and $\phi(z)$ is the first order formula obtained by replacing the constant symbol $d$ with a new variable $z$, we have that $\M \models \forall z (N(z) \rightarrow \lnot \phi(z))$. Now let us choose an elementary extension $\M' \succeq \M$ in which the interpretation of $N$ contains an element $d$ that is greater than all standard integers (this can be easily done by compactness). Then, if for $n \in N^{\M'}$  we write $R'_n$ for $\{\tuple a : (n, \tuple a) \in R^{\M'}\}$ and $S'_n$ for $\{\tuple a : (n, \tuple a) \in S^{\M'}\}$, since $\M'$ is an elementary extension of $\M$ it must be the case that $(M', R'_n) \in \D$ and $R'_n \subseteq S'_n \subseteq R'_{n+1}$ for all $n \in \mathbb N$. But then, as before, it must be the case that $(M', \bigcup_{n \in \mathbb N} S'_n) = (M', \bigcup_{n \in \mathbb N} R'_n) \in \D$; and since $n < d$ for all $n \in \mathbb N$, we must have that $\phi(d)$ is true in $\M'$, contradicting the fact $\M' \succeq \M$.  

It remains to verify that (\ref{eq:noD}) indeed states that there exists some family $I$ of indexes, all smaller than $d$, such that $(M, \bigcup_{n \in I} S_n) \in \D$. Suppose that this is the case, and let $X = \{ \varepsilon[n/i] : n \in I\}$. Then $\{\varepsilon\} \equiv_{\emptyset} X$ and $\M \models_X i < d$; and furthermore, for 
\[
    Y = X[\tuple M/\tuple v]_{| S i \tuple v} = \{s : \dom(s) = \{i\} \cup \Rng(\tuple v): s(i) \in I, s(\tuple v) \in S_i\},
\]
we have that $(M, Y(\tuple v)) = (M, \bigcup_{n \in I} S_n) \in \D$, and so $\M \models_X \forall \tuple v (S i \tuple v \hookrightarrow \D \tuple v)$ as required. 

Conversely, suppose that the initial team $\{\varepsilon\}$ satisfies  (\ref{eq:noD}). Then there exists a team $X$ with $\dom(X) = \{i\}$ such that $\M \models_X i < d \wedge \forall \tuple v (S i \tuple v \hookrightarrow \D \tuple v)$. Now let $I = X(i)$. I state that $I$ is as required: indeed, since $\M \models_X i < d$ we have that $n < d$ for all $n \in I$, and since $\M \models_X \forall \tuple v (S i \tuple v \hookrightarrow \D \tuple v)$ we have that, for $Y = X[\tuple M/\tuple v]_{|S i \tuple v]}$, $(M, Y(\tuple v)) \in \D$. But 
\[
    Y(\tuple v) = \{s(\tuple v) : s(i) \in I, s(\tuple v) \in S_i\} = \bigcup_{n \in I} S_i
\]
and this concludes the proof. 
\qed
\end{proof}
\begin{propo}
Let $\D(R)$ be a domain-independent, strongly first order DED, let $(A_1, R_1) \in \D$ and let $(A_1, R_1) \preceq (A_2, R_2)$ (that is, $(A_2, R_2)$ is an elementary extension of $(A_1, R_1)$). Then $(A_2, S_1) \in \D$ for all $S_1$ with $R_1 \subseteq S_1 \subseteq R_2$. 
\label{propo:noDjump}
\end{propo}
\begin{proof}
Suppose that it is not so: then we can build a chain as per Proposition \ref{propo:noChain}. Indeed, consider the theory  
\begin{align*}
    T := & \{\phi(R, \tuple a): \tuple a \text{ tuple in } A_2, \phi(R, \tuple a) \in \FO, (A_2, R_2) \models \phi(R_2, \tuple a)\} \cup\\
    &\{S \tuple b: \tuple b \in R_2\} \cup \{\forall \tuple x(S \tuple x \rightarrow R \tuple x), \lnot \D(S)\}
\end{align*}
where $\phi(R, \tuple a)$ ranges over all first order formulas with parameters in $A_2$ that do not contain $S$. By compactness, $T$ is satisfiable: indeed, every finite subtheory $T_0$ of $T$ will be entailed by some sentence of the form $\phi(R,\tuple a) \wedge \bigwedge_{i=1}^t S \tuple b_i \wedge \forall \tuple x (S \tuple x \rightarrow R \tuple x) \wedge \lnot \D(S)$, where $(A_2, R_2) \models \phi(R_2, \tuple a)$ and all the $\tuple b_i$ are in $R_2$. But then, since $(A_1, R_1) \preceq (A_2, R_2)$, we can find in $A_1$ tuples $\tuple c\tuple d_1 \ldots \tuple d_t$ with the same identity type as $\tuple a \tuple b_1 \ldots \tuple b_t$ such that $(A_1, R_1) \models \phi(R_1, \tuple c)$ and all the $\tuple d_i$ are in $R_1$. Then in $(A_2, R_2, S_1)$ we have that $\phi(R_2, \tuple c)$, because $(A_2, R_2)$ is an elementary extension of $(A_1, R_1)$; that every $\tuple d_i$ is in $S_1$, because it is in $R_1$ and $R_1 \subseteq S_1$; and that $S_1$ is contained in $R_2$ and $(A_2, S_1) \not\in \D$, by hypothesis.  Thus, $T_0$ can be satisfied by interpreting $\tuple a \tuple b_1 \ldots \tuple b_t$ as $\tuple c \tuple d_1 \ldots \tuple d_t$, $R$ as $R_2$, and $S$ as $S_1$. 

 Now, a model of $T$ describes an elementary extension $(A_3, R_3) \succeq (A_2, R_2)$ and some $S_2$ with $R_2 \subseteq S_2 \subseteq R_3$ and $(A_3, S_2) \not \in \D$. 

Iterating this construction, we obtain an infinite elementary chain $(A_1, R_1) \preceq (A_2, R_2) \preceq \ldots$ such that $(A_i, R_i) \in \D$ for all $i$, as well as relations $S_i$ such that $R_i \subseteq S_i \subseteq R_{i+1}$ and $(A_{i+1}, S_i) \not \in \D$ for all $i$. Now if we let $M  =\bigcup_i A_i$, by the domain independence of $\D$ we have that $(M, R_i) \in \D$, $(M, S_i) \not \in \D$, and $R_i \subseteq S_i \subseteq R_{i+1}$ for all $i \in \mathbb N$; but Proposition \ref{propo:noChain} shows that this is not possible if $\D$ is strongly first order. \qed
\end{proof}
Then by a straightforward application of compactness we get the following: 
\begin{coro}
Let $\D$ be a $k$-ary, domain-independent, strongly first order DED, let $(A, R) \in \D$ and let $\tuple a$ be a tuple of elements of $A$ and $(\tuple a^{(i)})_{i \in \mathbb N}$ be a sequence of $k$-tuples over $A$ such that 
\begin{enumerate}
\item $\tuple a^{(i)} \in R$ for all $i \in \mathbb N$; 
\item If $i \not = j$, $\tuple a^{(i)}$ and $\tuple a^{(j)}$ are disjoint apart from $\tuple a$; 
\item The identity types of $\tuple a^{(i)} \tuple a$ are the same for all $i \in \mathbb N$. 
\end{enumerate}
Furthermore , let $B \supseteq A$, and let $(\tuple b^{(i)})_{i \in \mathbb N}$ be a sequence of $k$-tuples over $B$ s.t. 
\begin{enumerate}
\item Every $\tuple b^{(i)}$ is disjoint from $A$ except on $\tuple a$; 
\item If $i \not = j$, $\tuple b^{(i)}$ and $\tuple b^{(j)}$ are disjoint except on $\tuple a$; 
\item The identity types of the $\tuple b^{(i)}\tuple a$ are all equal to the identity types of the $\tuple a^{(i)} \tuple a$. 
\end{enumerate}
Then $(B, R \cup \{\tuple b^{(i)} : i \in \mathbb N\}) \in \D$. 
\label{coro:add_bqs}
\end{coro}
\begin{proof}
Consider the theory 
\begin{align*}
  T:=&\{\phi(\tuple c) \in \FO : (A, R) \models \phi(\tuple c), \tuple c \text{ tuple in } A\} \cup\\ 
   &\{\tau_{\tuple b^{(1)} \ldots \tuple b^{(q)}\tuple c}(\tuple b^{(1)} \ldots \tuple b^{(q)} \tuple c) : q \in \mathbb N, \tuple c \text{ tuple in } A\} \cup \\
   &\{R \tuple b^{(q)} : q \in \mathbb N\}
\end{align*}
where $\tau_{\tuple b^{(1)} \ldots \tuple b^{(q)}\tuple c}(\tuple u^{(1)} \ldots \tuple u^{(q)} \tuple v)$ is the identity type of the tuple $\tuple b^{(1)} \ldots \tuple b^{(q)} \tuple c$. 

This theory is finitely satisfiable. Indeed, any finite subset of it will be entailed by some formula of the form 
\begin{equation}
    \phi(\tuple c) \wedge \tau(\tuple b^{(1)} \ldots \tuple b^{(q)}\tuple c) \wedge \bigwedge_{i=1}^q R \tuple b^{(i)}
    \label{eq:subth_bq}
\end{equation}
for some $q \in \mathbb N$, some tuple $\tuple c$ of elements of $A$ and some first order formula $\phi$ such that $(A, R) \models \phi(\tuple c)$, where $\tau$ is the identity type of $\tuple b^{(1)} \ldots \tuple b^{(q)} \tuple c$. Note that the $\tuple b^{(t)}$ may intersect $\tuple c$ only over $\tuple a$, because all elements of $\tuple c$ are in $A$. Then we can pick $q$ tuples $\tuple a^{(p_1)} \ldots \tuple a^{(p_q)}$ from our $\tuple a^{(i)}$ that are disjoint from $\tuple c$ apart from $\tuple a$: indeed, every element in $\Rng(\tuple c) \backslash \Rng(\tuple a)$ may appear in at most one $\tuple a^{(i)}$, because they are pairwise disjoint apart from $\tuple a$, and we have infinitely many of them. Then if we interpret $b^{(1)} \ldots b^{(q)}$ as $\tuple a^{(p_1)} \ldots \tuple a^{(p_q)}$ we have that (\ref{eq:subth_bq}) is satisfied in $A$: indeed, $\phi(\tuple c)$ holds and all $\tuple a^{(i)}$ are in $R$ by construction, and since like the $\tuple b^{(1)} \ldots \tuple b^{(q)}$ the $\tuple a^{(p_1)} \ldots \tuple a^{(p_q)}$ are pairwise disjoint outside from $\tuple a$, intersect $\tuple c$ only on $\tuple a$, and have the same identity type together with $\tuple a$, the identity types of $\tuple b^{(1)} \ldots \tuple b^{(q)} \tuple c$ and of $\tuple a^{(p_1)} \ldots \tuple a^{(p_q)} \tuple c$ are the same. 

Therefore, by compactness, there exists an elementary extension $(A', R') \succeq (A, R)$ with $\tuple b^{(1)}, \tuple b^{(2)}, \ldots \in R'$. Then by Proposition \ref{propo:noDjump} it must be the case that $(A', R \cup \{\tuple b^{(1)}, \tuple b^{(2)}, \ldots\}) \in \D$; and finally, since $\D$ is domain-independent we have that $(B,  R \cup \{\tuple b^{(1)}, \tuple b^{(2)}, \ldots) \in \D$ as well. \qed
\end{proof}
\begin{propo}
Let $\D$ be a $k$-ary domain-independent dependency. Suppose that there exist a domain $B = A \cup \bigcup_{i \in \mathbb N} B_i$, where all $B_i$ are disjoint from each other and from $A$, and relations $R \subseteq A^k$, $Q_i, T_i \subseteq (A \cup B_i)^k \backslash A^k$ such that 
\begin{enumerate}
\item For all $i, j \in \mathbb N$, there is an isomorphism $\mathfrak f_{i,j}: (A \cup B_i, R \cup Q_i, R \cup T_i) \rightarrow (A \cup B_j, R \cup Q_j, R \cup T_j)$ such that $\mathfrak f_{i,j}(a) = a$ for all $a \in A$; 
\item for all $I, J \subseteq \mathbb N$ and for $R_{I,J} = R \cup \bigcup_{i \in I} Q_i \cup \bigcup_{j \in J} T_j$
we have that 
\[
    (B, R_{I,J}) \in \D \text{ if and only if } I \cap J = \emptyset.
\]
\end{enumerate}
Then $\D$ is not strongly first order. 
\label{propo:DED_RQT}
\end{propo}
\begin{proof}
Without loss of generality, we can assume that $B \cap \mathbb N = \emptyset$. For any $\ell \in \mathbb N$, $\ell>1$, let $\M_\ell$ be a model with domain $B \cup \{1 \ldots \ell\}$, with a unary predicate $N$ with $N^{\M_\ell} = \{1 \ldots \ell\}$, with two constants $\textbf 1$ and $\textbf{end}$ with $\textbf 1^{\M_\ell} = 1$ and $\textbf{end}^{\M_\ell} = \ell$, with a binary relation $E$ such that $E^{\M_\ell} = \{(i, i+1): i \in 1 \ldots \ell-1\}$  and with two $(k+1)$-ary relations $Q$, $T$ with $Q^{\M_\ell} = \{(i, \tuple a): i \in 1\ldots \ell, \tuple a \in R \cup Q_i\}$
and $T^{\M_\ell} = \{(i, \tuple a): i \in 1\ldots \ell, \tuple a \in R \cup T_i\}$. Then the $\FO(\D)$ sentence
\begin{align*}
\forall n \forall n'( &(N(n) \wedge N(n')) \hookrightarrow \exists v \exists v'( (n=\textbf 1 \hookrightarrow v=\textbf 1) \wedge (n=\textbf{end} \hookrightarrow v\not =\textbf 1) \wedge\\
& (E(n,n') \hookrightarrow ((v=\textbf 1 \hookrightarrow v' \not = \textbf 1) \wedge (v \not = \textbf 1 \hookrightarrow v' = \textbf 1))) \wedge\\
&\forall \tuple z ( ((v = \textbf 1 \wedge Q n \tuple z) \vee (v \not = \textbf 1 \wedge Tn\tuple z)) \hookrightarrow \D \tuple z) \wedge \\
&\forall \tuple w ( ((v' = \textbf 1 \wedge Q n' \tuple w) \vee (v' \not = \textbf 1 \wedge Tn'\tuple w)) \hookrightarrow \D \tuple w) \wedge\\
&(n =n' \hookrightarrow v = v'))
)
\end{align*}
is true in $\M_\ell$ if and only if $\ell$ is even.\footnote{This is based on the Dependence Logic sentence to express even cardinality found in \cite{vaananen07}: in brief, for each $n \in 1 \ldots \ell$ we choose an index $v$ that is either $1$ or not so, and the dependence statements ensure that we cannot associate both $1$ and a different value for the same $n$; whenever some value $n$ is associated with $1$, its successor cannot be so; the index $1$ must be associated with $1$, and the last element $\ell$ cannot be so.} A standard back-and-forth argument shows that no first order sentence can be true in $\M_\ell$ if and only if $\ell$ is even; thus, $\FO(\D)$ must be more expressive than $\FO$, i.e. $\D$ is not strongly first order. 

Let us check that the sentence written above is true in $\M_\ell$ if and only if $\ell$ is even. Suppose that $\ell$ is even, and let $Y$ be the set of all assignments over the variables $n,n',v,v'$ such that  
\begin{itemize}
\item $s(n)$ and $s(n')$ are in $\{1 \ldots \ell\}$; 
\item $s(v) = 1$ if $s(n)$ is odd, and $s(v) = \ell$ otherwise; 
\item $s(v') = 1$ if $s(n')$ is odd, and $s(v') = \ell$ otherwise. 
\end{itemize}
Then 
\begin{itemize}
\item $\M_\ell \models_Y (n= \textbf 1 \hookrightarrow v=\textbf 1) \wedge (n=\textbf{end} \hookrightarrow v\not =\textbf 1)$, because for all $s \in Y$ if $s(n) = \textbf{1}^{\M_\ell} = 1$ then $s(v) = 1$ (as $1$ is odd) and if $n = \textbf{end}^{\M_\ell} = \ell$ then $s(v) = \ell \not = 1$  (as $\ell$ is even);
\item $\M_\ell \models_Y  E(n,n') \hookrightarrow ((v=\textbf 1 \hookrightarrow v' \not = \textbf 1) \wedge (v \not = \textbf 1 \hookrightarrow v' = \textbf 1))$: indeed, whenever $s(n') = s(n) + 1$ for $s \in Y$, by construction, if $s(v) = 1$ then $s(n)$ is odd, so $s(n')$ is even and $s(v')  = \ell \not = 1$, and if $s(v) \not= 1$ then $s(n)$ is even, and so $s(n')$ is odd and $s(v') = 1$;
\item $\M_\ell \models_Y \forall \tuple z ( ((v = \textbf 1 \wedge Q n \tuple z) \vee (v \not = \textbf 1 \wedge Tn\tuple z)) \hookrightarrow \D \tuple z)$: indeed, for 
\[
Z = Y[\tuple M / \tuple z]_{|(v = \textbf 1 \wedge Q n \tuple z) \vee (v \not = \textbf 1 \wedge Tn\tuple z)},
\]
we have that 
\begin{align*}
Z(\tuple z) &= \bigcup \{R \cup Q_i : \exists s \in Y, s(n) = i, s(v) = 1\} \cup\\
&~~~~~\bigcup \{R \cup T_i : \exists s \in Y, s(n) = i, s(v) \not = 1\}\\
&= \bigcup \{R \cup Q_i : i \in 1 \ldots \ell, i \text{ odd}\} \cup \bigcup \{R \cup T_i : i \in 1 \ldots \ell, i \text{ even}\}\\
& = R_{I,J} \text{ for } I = \{i \in 1 \ldots \ell, i \text{ odd}\} \text{ and } J = \{1 \ldots \ell\} \backslash I
\end{align*}
and so $(M_\ell, Z(\tuple z)) \in \D$ by hypothesis. 
\item $\M_\ell \models_Y \forall \tuple w ( ((v' = \textbf 1 \wedge Q n' \tuple w) \vee (v' \not = \textbf 1 \wedge Tn'\tuple w)) \hookrightarrow \D \tuple w)$: the argument is as in the previous point;
\item $\M_\ell \models_Y (n = n' \hookrightarrow v = v')$, because if $s \in Y$ is such that $s(n) = s(n')$ by construction we also have that $s(v) = s(v')$.
\end{itemize}
Now, for $X = \{\varepsilon\}[\tuple M/nn']_{|N(n) \wedge N(n')}$ we have that $X \equiv_{vv'} Y$; therefore our sentence is true in $\M_\ell$, as required. 

Conversely, let us suppose that the sentence is true in $\M_\ell$. Then there exists a team $Y$ such that $Y \equiv_{vv'}  \{\varepsilon\}[\tuple M/nn']_{|N(n) \wedge N(n')}$ and that $Y$ satisfies the part of our expression that follows $\exists v \exists v'$. Then for
\begin{align*}
I &= \{i \in 1 \ldots \ell: \exists s \in Y, s(n) = i, s(v) = 1\}; \\
J &= \{i \in 1 \ldots \ell: \exists s \in Y, s(n) = i, s(v) \not = 1\};\\
I' &= \{i \in 1 \ldots \ell: \exists s \in Y, s(n') = i, s(v') = 1\};\\
J' &= \{i \in 1 \ldots \ell: \exists s \in Y, s(n') = i, s(v') \not = 1\}
\end{align*}
We have that
\begin{itemize}
\item $1 \in I \backslash J$ and $\ell \in J \backslash I$, because $\M_\ell \models_Y (n=\textbf 1 \hookrightarrow v=\textbf 1) \wedge (n=\textbf{end} \hookrightarrow v\not =\textbf 1)$; 
\item $I \cap J = \emptyset$: indeed, for
\[
    Z = Y[\tuple M / \tuple z]_{|(v = \textbf 1 \wedge Q n \tuple z) \vee (v \not = \textbf 1 \wedge Tn\tuple z)},
\]
we have that 
\begin{align*}
Z(\tuple z) &= \bigcup \{R \cup Q_i : \exists s \in Y, s(n) = i, s(v) = 1\} \cup\\
&~~~~\bigcup \{R \cup T_i : \exists s \in Y, s(n) = i, s(v) \not = 1\}\\
&= R_{I, J}
\end{align*}

and so by hypothesis if $I \cap J \not = \emptyset$ then  $(M_\ell, Z(\tuple z)) \not \in \D$, while the formula states that $\M_\ell \models_Z \D \tuple z$; 
\item $I' \cap J' = \emptyset$: similar argument as in the previous point;
\item $I = I'$: indeed, if $i \in I$ then there exists some $s \in Y$ with $s(nv) = i1$. But then if we choose some $s' \in Y$ with $s'(nn') = ii$ (which must exist because $Y \equiv_{vv'}  \{\varepsilon\}[\tuple M/nn']_{|N(n) \wedge N(n')}$) we must still have that $s'(v) = 1$, because otherwise we would have that $i \in I \cap J$ and as we saw this is impossible. But as $\M_\ell \models_Y n = n' \hookrightarrow v = v'$ this has a consequence that $s'(v') = s'(v) = 1$ and so $i \in I'$. By a symmetric argument, if $i \in I'$ then $i \in I$, and so $I = I'$. 
\item $J = J'$: Similar to the previous point. If $i \in J$ then there is some $s \in Y$ with $s(n) = i, s(v) \not = 1$. Now consider an assignment $s' \in Y$ with $s(nn') = ii$: we must still have that $s'(v) \not= 1$ because otherwise $i \in I$ and as we saw $I \cap J = \emptyset$. But since $\M_\ell \models_Y n = n' \hookrightarrow v = v'$, we must have that $s'(v') = s'(v) \not = 1$, and so $i \in J'$. Similarly, we can see that if $i \in J'$ then $i \in J$. 
\item If $i \in I$ and $i < \ell$ then $i+1 \in J$. Indeed, if $i \in I$, there exists some $s \in Y$ with $s(nv) = i1$. Now consider any $s' \in Y$ with $s'(n) = i$ and $s'(n') = i+1$: it must be the case that $s'(v) = 1$ as well, because otherwise we would have that $i  \in I \cap J = \emptyset$. But $\M_\ell \models_Y E(n,n') \hookrightarrow ((v=\textbf 1 \hookrightarrow v' \not = \textbf 1) \wedge (v \not = \textbf 1 \hookrightarrow v' = \textbf 1))$, and therefore $s'(v') \not = 1$, i.e. $i+1 \in J' = J$. 
\item If $i \in J$ and $i < \ell$ then $i+1 \in I$: similar to the previous point. If $i \in J$, there exists some $s \in Y$ with $s(n) = i$, $s(v) \not = 1$. Now take some $s' \in Y$ with $s'(n) = i$, $s'(n') = i+1$: then $s'(v) \not = 1$ as well, since otherwise $i \in I \cap J = \emptyset$. But then $s'(v') = 1$, and so $i+1 \in I' = I$. 
\end{itemize}

Therefore, the set $I$ contains $1$; whenever $i \leq \ell-2$ and $i \in I$, $i+1 \in J$, and so $i+2  \in I$; and $\ell$ is not in $I$. This is possible only if $\ell$ is even, as required. 

Finally, we must verify that no first order sentence can be true in $\M_\ell$ if and only if $\ell$ is even.

To do so it suffices to show that, for all $n \in \mathbb N$ and for $\ell > 2^{n+1}$, no first order sentence of quantifier rank up to $n$ can tell apart the models $\M_{\ell}$ and $\M_{\ell+1}$, i.e. Duplicator has a winning strategy in the $n$-moves Ehrenfeucht-Fra\"iss\'e game $EF_n(\M_{\ell}, \M_{\ell+1})$. We can assume that this game starts with all elements of $A$ already played in $\M_{\ell}$ and answered with the same elements in $\M_{\ell+1}$, and that Spoiler will never choose elements of $A$ again in either model.

Let us proceed by induction on $n$: 

\begin{description}
\item[Base Case:] Let us consider the one-move game $EF_1(\M_{\ell}, \M_{\ell+1})$ where $\ell > 4$.  

If Spoiler plays in $\M_{\ell}$ and chooses an index $t \in \{1 \ldots \ell\}$ or an element $b \in B_t$ for some $t$, Duplicator can choose an index $t'$ that is $1$, $2$, $\ell$ or or $\ell + 1$ if and only if $t$ is $1,2, \ell-1$ or $\ell$ respectively, and then play $t'$ or $\mathfrak f_{t,t'}(b)$ respectively. 

If instead Spoiler plays in $\M_{\ell + 1}$ and chooses an index $t' \in \{1 \ldots \ell+1\}$ or an element $b' \in B_t$ for some such $t'$, Duplicator can likewise choose an index $t$ that is $1$, $2$, $\ell-1$ or $\ell$ if and only if $t'$ is $1$,$2$, $\ell$ or $\ell+1$ respectively, and then play $t$ or $\mathfrak f_{t,t'}^{-1}(b)$ respectively. 

In either case, Duplicator then wins the game: 
\begin{itemize}
\item The strategy only associates elements in $\{1 \ldots \ell\}$ to elements in $\{1\ldots \ell+1\}$ and vice versa, so the unary predicate $N$ is respected; 
\item The strategy associates endpoints to endpoints, so atomic formulas of the form $x = \textbf{1}$ or $x = \textbf{end}$ are respected; 
\item Atoms of the form $E(x,\textbf{1})$ or $E(\textbf{end},x)$ are never satisfied in either model, so there is nothing to check about them; 
\item Atoms of the form $E(\textbf{start},x)$ are satisfied (in $\M_\ell$ or $\M_{\ell+1}$) only if $x=2$, and this strategy maps $2$ in $\M_\ell$ to $2$ in $\M_{\ell+1}$ and vice versa; 
\item Atoms of the form $E(x, \textbf{end})$ are satisfied in $\M_\ell$ only if $x = \ell-1$ and in $\M_{\ell+1}$ only if $x = \ell$, and this strategy maps $\ell-1$ in $\M_{\ell}$ to $\ell$ in $\M_{\ell+1}$;
\item In one move it is not possible to pick two indexes that are one the successor of each other (note that the elements of $A$ never occur in the relation $E$), so atoms of the form $Exy$ are trivially respected; 
\item The $\mathfrak f$s are isomorphisms that keep $A$ fixed pointwise, so the relations $Q$ and $T$ are respected. 
\end{itemize}
\item[Induction Case]: Consider now the game $EF_{n+1}(\M_{\ell}, \M_{\ell+1})$ where $\ell > 2^{n+2}$. We must consider a few different cases depending on the first move of Spoiler:  
\begin{enumerate}
\item If Spoiler plays in $\M_\ell$ and picks an element $t \in \{1 \ldots 2^{n+1}+1\}$ or an element $b \in B_t$ for some $t \in \{1 \ldots 2^{n+1}+1\}$, answer with the same element in $\M_{\ell+1}$. Now, the submodels of these models that are associated to indexes no greater than $t$ are isomorphic to each other (their domains are both of the form $A \cup \{1 \ldots t\} \cup \bigcup \{B_1 \ldots B_t\}$), so for the rest of the game Duplicator can play along this isomorphism for these submodels. The parts of the two models that instead \emph{begin} from $t$ (included) are of the form $A \cup \{t \ldots \ell\} \cup \bigcup\{B_{t} \ldots B_{\ell}\}$ and $A \cup \{t \ldots \ell+1\} \cup \bigcup\{B_{t} \ldots B_{\ell+1}\}$ respectively;  but these are isomorphic to $\M_{\ell - t+1}$ and $\M_{\ell-t+2}$ respectively (it is just a matter of gluing together the various isomorphisms from $A \cup B_{t}$ to $A \cup B_1$, from $A \cup B_{t+1}$ to $A \cup B_2$ et cetera, which we can do because they all agree over $A$). But $\ell - t + 1> 2^{n+2} - (2^{n+1}+1) +1 = 2^{n+1}$, so by induction hypothesis Duplicator can survive for $n$ turns by playing elements between these submodels.\footnote{Note that both sub-strategies associate $t$ to $t$ and $B_t$ to $B_t$, so there is no conflict between them insofar as they overlap; and that no atomic formula holds between an element that is picked according to one substrategy and one that doesn't, so if Duplicator can win both subgames she can win the game.}
\item If Spoiler plays in $\M_{\ell + 1}$ and picks an element $t \in \{1 \ldots 2^{n+1}+1\}$ or an element $b \in B_t$ for some $t \in \{1 \ldots 2^{n+1}+1\}$, answer with the same element in $\M_{\ell}$. By the same argument used above, Duplicator can then survive for $n$ more turns. 
\item If Spoiler plays in $\M_\ell$ and picks an element $t \in \{2^{n+1}+2 \ldots \ell\}$, or an element $b \in B_t$ for some $t \in \{2^{n+1}+2 \ldots \ell\}$, answer with $t+1$ or $\mathfrak f_{t,t+1}(b)$ respectively. 

This time, the submodels associated to indexes starting from $t$ and $t+1$ respectively are isomorphic to each other (and to $A \cup \{1 \ldots \ell - t+1\} \cup \bigcup\{B_1 \ldots B_{\ell-t+1}\}$); and the ones associated to indexes up to $t$ are equal to $\M_{t}$ and $\M_{t+1}$ respectively, and since $t > 2^{n+1}$ by induction hypothesis Duplicator can survive for $n$ turns between these two submodels. 
\item If Spoiler plays in $\M_{\ell+1}$ and picks an element $t \in \{2^{n+1}+2 \ldots \ell+1\}$, or an element $b \in B_t$ for some $t \in \{2^{n+1}+2 \ldots \ell+1\}$, answer with $t-1$ or $\mathfrak f_{t-1,t}^{-1}(b)$ respectively. By the same argument used above, Duplicator can then survive for $n$ more turns.
\end{enumerate}
\end{description}
\qed
\end{proof}

\subsection*{The Characterization}
The following notions of U-sentences and U-embeddings are from \cite{galliani2022strongly}, in which they were used to characterize strongly first order union-closed dependencies:
\begin{defin}[U-sentences, $\Rrightarrow_U$]
Let $R$ be a $k$-ary relation symbol and let $\tuple a$ be a tuple of constant symbols. Then a first order sentence over the signature $\{R, \tuple a\}$ is a U-sentence if and only if it is of the form $\exists \tuple x (\eta(\tuple x) \wedge \forall \tuple y (R \tuple y \rightarrow \theta(\tuple x, \tuple y)))$, 
where $\tuple x$ and $\tuple y$ are disjoint tuples of variables without repetitions, $\eta(x)$ is a conjunction of first order literals over the signature $\{R, \tuple a\}$ in which $R$ occurs only positively,  and  $\theta(\tuple x, \tuple y)$ is a first order formula over the signature $\{\tuple a\}$  (i.e. in which $R$ does not appear). Given two models $\A$ and $\B$ with the same signature, we will write $\A \Rrightarrow_U \B$ if, for every $U$-sentence $\phi$, $\A \models \phi \Rightarrow \B \models \phi$. 
\end{defin}
\begin{propo}
U-sentences are closed by conjunction, i.e. if $\phi$ and $\psi$ are U-sentences then $\phi \wedge \psi$ is equivalent to some U-sentence. 
\label{propo:U-sent-conj}
\end{propo}
\begin{proof}
Let 
\[
    \phi := \exists \tuple x(\eta(\tuple x) \wedge \forall \tuple y(R \tuple y \rightarrow \theta(\tuple x, \tuple y)))
\]
and 
\[
    \psi := \exists \tuple z(\eta'(\tuple z) \wedge \forall \tuple w(R \tuple w \rightarrow \theta'(\tuple z, \tuple w))) 
\]
where, up to variable renaming, we can assume that $\tuple x \tuple y$ and $\tuple z \tuple w$ have no variables in common. 

Then $\phi \wedge \psi$ is equivalent to 
\[
    \exists \tuple x \tuple z( (\eta(\tuple x) \wedge \eta'(\tuple z)) \wedge \forall \tuple y(R \tuple y \rightarrow (\theta(\tuple x, \tuple y) \wedge \theta'(\tuple z, \tuple y))))
\]
which is of the required form. \qed
\end{proof}
\begin{propo}
Let $\phi = \exists \tuple x(\eta(\tuple x) \wedge \forall \tuple y(R \tuple y \rightarrow \theta(\tuple x, \tuple y)))$ be a U-sentence. Then there exists a $\FO(=\!\!(\cdot), \NE)$ formula $\phi'(\tuple y)$ over the empty signature, with free variables in $\tuple y$, such that $\M \models_X \phi'(\tuple y) \Leftrightarrow (M, X(\tuple y)) \models \phi$
for all $\M$ and $X$.
\label{Usent2CNE}
\end{propo}
\begin{proof}
Take $\phi'(\tuple y) := \exists \tuple x( =\!\!(\tuple x) \wedge \eta'(\tuple x) \wedge \theta(\tuple x, \tuple y))$,
where $\eta'$ is obtained from $\eta$ by replacing every atom $R \tuple z$ (for $\tuple z \subseteq \tuple x$) with $\tuple z = \tuple z  \vee (\NE(\tuple z) \wedge \tuple z = \tuple y)$. 

Let us verify that $\phi'$ satisfies the required condition. 

If $\M \models_X \phi'(\tuple y)$, there exists some $Y \equiv_{\tuple y} X$ such that $\M \models_Y =\!\!(\tuple x) \wedge \eta'(\tuple x, \tuple y) \wedge \theta(\tuple x, \tuple y)$. 
\begin{itemize}
\item Since $\M \models_Y =\!\!(\tuple x)$, there exists one tuple of elements $\tuple m$ such that $s(\tuple x) = \tuple m$ for all $s \in Y$; 
\item Since $\M \models_Y \eta'(\tuple x)$, for every literal $x_i = x_j$ occurring in $\eta$ we have that $m_i = m_j$, and likewise for literals in $\eta$ of the form $x_i \not = x_j$, $x_i = c$ or $c = d$ where $c$ and $d$ are constants; and for every literal of the form $R \tuple z$ for $\tuple z \subseteq \tuple x$, we have that 
\[
    M \models_Y (\tuple z = \tuple z) \vee (\NE(\tuple z) \wedge \tuple z = \tuple y)
\]
and so $Y = Y_1 \cup Y_2$ where $Y_2 \not = \emptyset$ and $s(\tuple z) = s(\tuple y) \in Y_2(\tuple y) \subseteq Y(\tuple y) = X(\tuple y)$ for all $s \in Y_2$. So $\tuple z[\tuple m/\tuple x] = s(\tuple z) \in X(\tuple y)$, as required, and in conclusion $(M, X(\tuple y)) \models \eta(\tuple m)$; 
\item Consider any $\tuple a \in X(\tuple y) = Y(\tuple y)$. Since $s(\tuple x) = \tuple m$ for all $s \in Y$, there exists some $s \in Y$ with $s(\tuple x \tuple y) = \tuple m \tuple a$; and since $\M \models_Y \theta(\tuple x, \tuple y)$, by Proposition \ref{propo:flat} we have that $\theta(\tuple m, \tuple a)$. Therefore, $(M, X(\tuple y)) \models \forall \tuple y(R \tuple y \rightarrow \theta(\tuple m, \tuple y))$. 
\end{itemize}
So in conclusion $(M, X(\tuple y)) \models \exists \tuple x(\eta(\tuple x) \wedge \forall \tuple y(R \tuple y \rightarrow \theta(\tuple x, \tuple y)))$, as required. 

Conversely, suppose that $(M, X(\tuple y)) \models \exists \tuple x(\eta(\tuple x) \wedge \forall \tuple y(R \tuple y \rightarrow \theta(\tuple x, \tuple y)))$. Then there exists some $\tuple m$ such that  $(M, X(\tuple y)) \models \eta(\tuple m) \wedge \forall \tuple y(R \tuple y \rightarrow \theta(\tuple m, \tuple y))$. Now let $Y = X[\tuple m/\tuple x] = \{s[\tuple m/\tuple x] : s \in X\}$. Clearly $Y \equiv_{\tuple y} X$ and $\M \models_Y =\!\!(\tuple x)$. Additionally, $\M \models_Y \eta'(\tuple x)$: indeed, every identity literal occurring in $\eta$ occurs unchanged in $\eta'$ too, and if an atom $R \tuple z$ (for $\tuple z \subseteq \tuple x$) occurs in $\eta$ then there exists some $s \in Y$ with $s(\tuple y) = s(\tuple z)$, and so we can see that $\M \models_Y (\tuple z = \tuple z) \wedge (\NE(\tuple z) \wedge \tuple z = \tuple y)$ by splitting $Y$ as $Y = Y \cup \{s\}$. Finally, $\M \models_Y \theta(\tuple x, \tuple y)$ by Proposition \ref{propo:flat}, because for every $s \in Y$ we have that $s(\tuple y) \in X(\tuple y)$ and $s(\tuple x) = \tuple m$, and $\M \models \theta(\tuple m, s(\tuple y))$. 

This concludes the proof. 
\qed
\end{proof}
\begin{defin}[U-embedding]
A structure $(A, R)$ is said to be \emph{U-embedded} in a structure $(B, S)$ if 
\begin{enumerate}
\item $(A, R)$ is a substructure of $(B, S)$; 
\item For every finite tuple of parameters $\tuple a$ in $A$ and every first order formula $\theta(\tuple y, \tuple z)$ over the empty signature, \\$(A, R) \models \forall \tuple y (R \tuple y \rightarrow \theta(\tuple y, \tuple a)) \Rightarrow (B, S) \models \forall \tuple y (S \tuple y \rightarrow \theta(\tuple y, \tuple a))$.\\
If a structure $(A, R)$ is isomorphic to some structure $(A', R')$ which is U-embedded in $(B, S)$, we say that the isomorphism $\iota: (A, R) \rightarrow (A', R')$ is a \emph{U-embedding} of $(A, R)$ into $(B, S)$. 
\end{enumerate}
\end{defin}
The next proposition is proved as in the first part of Proposition 5 of \cite{galliani2022strongly}: 
\begin{propo}
Let $\A = (A, R)$ and let $\B = (B, S)$, where $R$ and $S$ are $k$-ary relations and $A$ is countably infinite, and suppose that $\A \Rrightarrow_U \B$. Then there exist an elementary extension $\B' \succeq\B$ and an U-embedding $\iota: \A \rightarrow \B'$.
\label{propo:U-embed}
\end{propo}
\begin{proof}
Let $\B'$ be an $\omega$-saturated elementary extension of $\B$ and let $(a_i)_{i \in \mathbb N}$ enumerate $A$. Then let us define by induction a sequence $(b_i)_{i \in \mathbb N}$ of elements of $B'$ such that 
\[
    (A, R, a_1 \ldots a_t) \Rrightarrow_U (B', S', b_1 \ldots b_t)
\]
for all $t \in \mathbb N$.
\begin{itemize}
\item \textbf{Base Case:} Since by hypothesis $(A, R) \Rrightarrow_U (B, S)$ and $(B', S')$ is an elementary extension of $(B, S)$, we have at once that $(A, R) \Rrightarrow_U (B', S')$. 
\item \textbf{Induction Case:} Suppose that $(A, R, a_1 \ldots a_t) \Rrightarrow_U (B', S', b_1 \ldots b_t)$. Now consider the element $a_{t+1}$ and the set of formulas 
\begin{align*}
    T(v) := \{&\exists \tuple x(\eta(S, \tuple x, b_1 \ldots b_t, v) \wedge \forall \tuple y(S \tuple y \rightarrow \theta(\tuple x, \tuple y, b_1 \ldots b_t, v))) : \\
    &(A, R) \models \exists \tuple x(\eta(R, \tuple x, a_1 \ldots a_t, a_{t+1}) \wedge \forall \tuple y(S \tuple y \rightarrow \theta(\tuple x, \tuple y, a_1 \ldots a_t, a_{t+1})))\}
\end{align*}
where $\eta$ ranges over conjunctions all first order literals in which the relation symbol $R$ occurs only positively and $\theta$ ranges over first order formulas in which $R$ does not occur at all. 

Then $T$ is finitely satisfiable: indeed, every finite subset of it is logically equivalent to a formula of the form 
\[
    \phi(S, b_1 \ldots b_t, v) = \exists \tuple x (\eta'(S, \tuple x, b_1 \ldots b_t, v) \wedge \forall \tuple y(S \tuple y \rightarrow \theta'(\tuple x, \tuple y, b_1 \ldots b_t, v)))
\]
such that $\eta'$ is a conjunction of literals in which $S$ appears only positively, $\theta'$ is a first order formula in which $S$ does not appear, and
\[
    (A, R) \models \phi(R, a_1 \ldots a_t, a_{t+1}).
\]
But then it is also the case that $(A, R) \models \exists v \phi(R, a_1 \ldots a_t, v)$; and since $\exists v \phi(R, a_1 \ldots a_t, v)$ is a U-sentence over the signature $\{R, a_1 \ldots a_t\}$, and\\$(A, R, a_1 \ldots a_t) \Rrightarrow_U (B, S, b_1 \ldots b_t)$, it follows that there exists some element $b' \in B'$ such that $(B, S) \models \phi(S, b_1 \ldots b_t, b')$. 

Therefore, since $\B'$ is $\omega$-saturated, there must exist some element $b_{t+1}$ such that $(B, S, b_1 \ldots b_t) \models T(b_{t+1})$; and therefore,\\$(A, R, a_1 \ldots a_{t}, a_{t+1}) \Rrightarrow_U (B, S, b_1 \ldots b_t, b_{t+1})$. 
\end{itemize}
Now let $\iota: A \rightarrow B'$ be such that $\iota(a_i) = b_i$ for all $i \in \mathbb N$. This $\iota$ is an injective function, since $a_i \not = a_j$ is a $U$-sentence for all $i, j \in \mathbb N$, and therefore we can let $(A', R') = \iota(A, R)$ be our isomorphic copy of $(A, R)$. It remains to show that it satisfies the required properties. 

Let $\tuple a = a_{i_1} \ldots a_{i_k}$ be a tuple of $k$ elements of $A$, and let $\tuple b = \iota(\tuple a) = b_{i_1} \ldots b_{i_k}$. If $\tuple a \in R$, then since $R \tuple a_{i_1} \ldots \tuple a_{i_k}$ is a U-sentence it must be the case that $\tuple b \in S$; and if instead $\tuple a \not \in R$, since $\forall \tuple y (R \tuple y \rightarrow (\tuple y \not = \tuple a))$ is a U-sentence it must be the case that $\tuple b \not \in S$. Therefore, $(A', R')$ is indeed a substructure of $(B, S)$.

Now suppose that $(A', R') \models \forall \tuple y(R' \tuple y \rightarrow \theta(\tuple y, \tuple{a'}))$ where $\tuple {a'} = (a'_{i_1} \ldots a'_{i_t})$ is a tuple of elements of $A' = \iota(A)$ and $\theta$ is first order formula over the empty signature. Then, for $\tuple a = (a_{i_1} \ldots a_{i_t})$, since $(A, R, \tuple a)$ and $(A', R', \tuple {a'})$ are isomorphic $(A, R) \models \forall \tuple y(R \tuple y \rightarrow \theta(\tuple y, \tuple a))$; and therefore, since this is a U-sentence, $(B', S') \models \forall \tuple y(S' \tuple y \rightarrow \theta(\tuple y, \tuple {a'}))$ as required. \qed
\end{proof}
\begin{coro}
Let $\A = (A, R)$ and let $\B = (B, S)$, where $R$ and $S$ are $k$-ary relations and $A$ is countably infinite, and suppose that $\A \Rrightarrow_U \B$. Then there exists a countably infinite structure $(B_0, S_0) \equiv (B, S)$ and an U-embedding $\iota: (A, R) \rightarrow (B_0, S_0)$. 
\label{coro:U-embed-count}
\end{coro}
\begin{proof}
By Proposition \ref{propo:U-embed}, we can find a (possibly uncountable) $(B', S') \succeq (B, S)$ and an U-embedding $\iota: (A, R) \rightarrow (B', S')$. Let $(A', R')$ be the isomorphic image of $(A, R)$ along $\iota$. Now, $A'$ is countable; therefore, by L\"owenheim-Skolem applied to $(B',S', (a')_{a' \in A'})$, there exists some countably infinite elementary substructure $(B_0, S_0, (a')_{a' \in A'}) \preceq (B', S', (a')_{a' \in A'})$ that still contains $(A', R')$ as a substructure. $(A', R')$ is U-embedded in $(B_0, S_0)$: indeed, if $(A', R') \models \forall \tuple y(R' \tuple y \rightarrow \theta(\tuple y, \tuple {a'}))$ then $(B', S') \models \forall \tuple y(S \tuple y \rightarrow \theta(\tuple y, \tuple {a'}))$, and so $(B_0, S_0) \models \forall \tuple y(S_0 \tuple y \rightarrow \theta(\tuple y, \tuple {a'}))$ as well. Therefore $\iota$ is also a U-embedding of $(A, R)$ into $(B_0, S_0)$, as required. \qed
\end{proof}

\begin{lem}
Let $(A, R)$ be countably infinite and U-embedded in $(B, S)$, and let $\tuple b \in (S \backslash R)^k$. Also, let $\tuple a$ list some finite $C$ such that $A \cap \Rng(\tuple b)\subseteq C \subseteq A$, and let $\tau(\tuple x, \tuple y)$ be the identity type of $\tuple b \tuple a$. Then there are infinitely many $\tuple a^{(1)}, \tuple a^{(2)}, \ldots$ in $R$ such that
\begin{enumerate}
\item All tuples $\tuple a^{(q)}$ satisfy $\tau(\tuple a^{(q)}, \tuple a)$; 
\item If $q \not = q'$ then $\Rng(\tuple a^{(q)}) \cap \Rng(\tuple a^{(q')}) \subseteq \Rng(\tuple a)$. 
\end{enumerate}
\label{lemma:infbi}
\end{lem}
\begin{proof}
Suppose that this is not the case: then there exist in $(A, R)$ a finite number of tuples $\tuple a^{(1)} \ldots \tuple a^{(q)}$ such that all $\tuple d \in R$ that satisfy $\tau(\tuple d, \tuple a)$, intersect one of them somewhere other than in $\tuple a$. So, if $\tuple c$ lists $\bigcup_{j = 1}^q \Rng(\tuple a^{(j)}) \backslash \Rng(\tuple a)$,
$(A, R) \models \forall \tuple y (R \tuple y \rightarrow (\tau(\tuple y, \tuple a)\rightarrow \tuple y \cap \tuple c \not = \emptyset))$ 
where $\tuple y \cap \tuple c \not = \emptyset$ is a shorthand for $\bigvee_{i,j} y_i = c_j$. Then, by the definition of U-embedding, $(B, S) \models \forall \tuple y (S \tuple y \rightarrow (\tau(\tuple y, \tuple a)\rightarrow \tuple y \cap \tuple c \not = \emptyset))$; and this is impossible, because $\tuple b$ is disjoint from $A$ except on $\tuple a$. \qed
\end{proof}
\begin{propo}
Let $(A, R)$ be countably infinite and U-embedded in $(B, S)$, let $t \in \mathbb N$, let $\tuple b_1, \ldots, \tuple b_t \in S \backslash R$, and suppose that $(A, R) \in \D$ where $\D$ is a domain-independent, strongly first order DED. Then $(B, R \cup \{\tuple b_1 \ldots \tuple b_t\}) \in \D$. 
\label{propo:noTuples}
\end{propo}
\begin{proof}
Suppose that this is not the case. Then let $\tuple a$ list $A \cap \bigcup_{i=1}^t \Rng(\tuple b_i)$, and for each $i=1\ldots t$ let $\tau_i(\tuple z, \tuple w)$ be the identity type of $\tuple b_i \tuple a$. Because of Lemma \ref{lemma:infbi}, $R$ contains infinitely many copies of each $\tuple b_i$, all satisfying the same identity types with $\tuple a$ and disjoint from any other copy of the same $\tuple b_i$ except over $\tuple a$. Now, for all $q \in \mathbb N$, let $B_q$ be a isomorphic, disjoint copy of $B\backslash A$, and let us identify $B_1$ with $B\backslash A$ itself; and for every $q$, let $(\tuple b_1^{(q)} \ldots \tuple b_t^{(q)})$ be the copy of $(\tuple b_1 \ldots \tuple b_t)$ in $A \cup B_q$. Finally, let $C = A \cup \bigcup_q B_q$. 
%

By Corollary \ref{coro:add_bqs}, $(C, R \cup \{\tuple b_i^{(q)} : q \in \mathbb N\}) \in \D$ for all $i \in 1 \ldots t$.

On the other hand, it cannot be that, for $R' = R \cup \{\tuple b_1^{(q)}, \dots, \tuple b_t^{(q)} : q \in \mathbb N\}$, $(C, R') \in \D$: indeed, otherwise by Corollary \ref{coro:DEDhom} we would have that $(C, R \cup \{\tuple b_1^{(1)} \ldots \tuple b_t^{(1)}\}) \in \D$ and so by domain independence $(B, R \cup \{b_1 \ldots b_t\}) \in \D$.

Therefore, there must exist a minimal $r \in 2 \ldots t$ such that, for $R' = R \cup \{\tuple b_1^{(q)} \ldots \tuple b_r^{(q)} : q \in \mathbb N\}$, $(C, R') \not \in \D$. Then let $\mathfrak g: \mathbb N \times \mathbb N \rightarrow \mathbb N$ be any bijection from $\mathbb N \times \mathbb N$ to $\mathbb N$ and,  for all $n \in \mathbb N$, let $Q_n = \{\tuple b_1^{\mathfrak g(n, q')} \ldots \tuple b_{r-1}^{\mathfrak g(n, q')}: q' \in \mathbb N\}$ and $T_n = \{\tuple b_r^{\mathfrak g(n, q')} : q' \in \mathbb N\}$. Let us see if we can apply Proposition \ref{propo:DED_RQT}.
%

By construction, it is clear there exist isomorphisms 
\[
\mathfrak f_{n,n'}: (A \cup \bigcup_{q' \in \mathbb N} B_{\mathfrak g(n,q')}, Q_n, T_n) \rightarrow (A \cup \bigcup_{q' \in \mathbb N} B_{\mathfrak g(n',q')}, Q_{n'}, T_{n'})
\]
that keep $A$ fixed pointwise. Now let $R_{I, J} = R \cup \bigcup_{i \in I} Q_i \cup \bigcup_{j \in J} T_j$ for $I, J \subseteq \mathbb N$. 
\begin{itemize}
\item If $I \cap J \not = \emptyset$ then $(C, R_{I,J}) \not \in \D$: otherwise, since $Q_i \cup T_i \subseteq R_{I,J}$ for some $i$, by Corollary \ref{coro:DEDhom} $(C, R \cup Q_i \cup T_i) \in \D$, which is impossible because this is isomorphic to $(C, R \cup \{\tuple b_1^q, \ldots, \tuple b_r^q: q \in \mathbb N\})$. 
\item If $I \cap J = \emptyset$ then $(C, R_{I, J}) \in \D$. Indeed, consider $R_I = R \cup \bigcup_{i \in I} Q_i$. $\bigcup_{i \in I} Q_i$ is a countably infinite set of copies of $b_1 \ldots b_{r-1}$, disjoint from each other and from $A$ outside of $\tuple a$; therefore, because of the minimality of $r$ we have that $(C, R_I) \in \D$. Now let $A' = A \cup \bigcup \{B_{\mathfrak g(n,q')} : n \in I, q' \in \mathbb N\}$. By the domain independence of $\D$, $(A', R_I) \in \D$; and by construction and by the fact that $I \cap J = \emptyset$, all tuples in $\bigcup_{j \in J} T_j$ are disjoint from $A'$ except on $\tuple a$. 

Since $R_I \supseteq R$, $R_I$ contains already infinitely many copies of $\tuple b_r$ disjoint from each other apart from $\tuple a$; and so, by Corollary \ref{coro:add_bqs}, $(C, R_{I,J}) \in \D$.

\end{itemize}
Thus, it is not possible for $\D$ to be strongly first order, which contradicts our hypothesis; and therefore, it must be the case that $(B, R \cup \{\tuple b_1 \ldots \tuple b_t\}) \in \D$. \qed
\end{proof}
\begin{coro}
Let $\D$ be a domain-independent, strongly first order DED, and suppose that $(A, R) \in \D$ is countably infinite and that $(A, R) \Rrightarrow_U (B, S)$. Then $(B, S) \in \D$. 
\label{coro:U-embed-preserve}
\end{coro}
\begin{proof}
Suppose that this is not the case. Then by Corollary \ref{coro:U-embed-count} there is  a countable model $(B_0, S_0) \equiv (B, S)$ and an U-embedding $\iota$ of $(A, R)$ into $(B_0, S_0) \not \in \D$. Let $(A_0, R_0)$ be the image of $(A, R)$ along $\iota$, and let $(\tuple b_j)_{j \in \mathbb N}$ enumerate $S_0 \backslash R_0$. Since $\D$ is domain-independent, we have that $(B_0, R_0) \in \D$; and since by Proposition \ref{propo:DEDuc} $\D$ is closed by unions of chains and $(B_0, S_0) \not \in \D$, there exists some $t \in \mathbb N$ such that $(B_0, R_0 \cup \{\tuple b_1 \ldots \tuple b_t\}) \not \in \D$. This contradicts Proposition \ref{propo:noTuples}. \qed
\end{proof}
\begin{propo}
Let $\D$ be a domain-independent, strongly first order DED and suppose that $(A, R) \in \D$. Then there exists a $U$-sentence $\phi$ such that 
\begin{itemize}
\item $(A, R) \models \phi$; 
\item $\phi \models \D(R)$. 
\end{itemize}
\label{propo:char_local}
\end{propo}
\begin{proof}
If $(A, R)$ is finite then it is easy to find a U-sentence that fixes R up to isomorphism (just list all tuples in $R$ and state that the relation contains precisely them) and that therefore, by the domain independence of $\D$, entails $\D(R)$. Otherwise, by L\"owenheim-Skolem we can assume that $(A, R)$ is countably infinite. Now let $(B_i, S_i)_{i \in I}$ list all countable models such that $(B_i, S_i) \not \in \D$. For all $i \in I$, it cannot be the case that $(A, R) \Rrightarrow_U (B_i, S_i)$: otherwise, by Corollary \ref{coro:U-embed-preserve} we would have that $(B_i, S_i) \in \D$ as well. So there exists some U-sentence $\phi_i(R)$ such that $(A, R) \models \phi_i(R)$ but $(B_i, S_i) \not \models \phi_i(S_i)$.  Now consider the theory $\{\phi_i(R) : i \in I\} \cup \{\lnot \D(R)\}$. 
This theory is unsatisfiable: if it had a model, by L\"owenheim-Skolem it should have a countable model - i.e. some $(B_i, S_i)$, which cannot be true because $(B_i, S_i) \not \models \phi_i(S_i)$. Thus, by compactness, there must exist some finite subtheory of it that is already unsatisfiable. Therefore, for some finite subset $I_0$ of $I$ we have that $(A, R) \models \bigwedge_{i \in I_0} \phi_i(R)$ and that $\bigwedge_{i \in I_0} \phi_i(R) \models \D(R)$; and since by Proposition \ref{propo:U-sent-conj} U-sentences are closed by conjunction, the result follows.  \qed
\end{proof}
\begin{propo}
Let $\D$ be a strongly first order, domain-independent DED. Then $\D(R)$ is logically equivalent to a finite disjunction of U-sentences. 
\label{propo:char_global}
\end{propo}
\begin{proof}
Let $(A_i, R_i)_{i \in I}$ list all countable models such that $(A_i, R_i) \in \D$. Then, by Proposition \ref{propo:char_local}, for every $i$ there exists some U-sentence $\phi_i(R)$ such that $(A_i, R_i) \models \phi_i(R_i)$ and $\phi_i(R) \models \D(R)$. Now consider the theory $\{\lnot \phi_i(R): i \in I\} \cup \{\D(R)\}$: 
this theory is unsatisfiable, because if it had a model it would need to have a countable model -- i.e. one of the $(A_i, R_i)$. So there is a finite subset $I_0 \subseteq I$ such that  $\D(R) \models \bigvee_{i \in I_0} \phi_{i}$, and so $\D(R)$ is equivalent to $\bigvee_{i \in I_0} \phi_{i}$. \qed
\end{proof}
\begin{thm}
Let $\D$ be a domain-independent DED. Then the following are equivalent:
\begin{enumerate}
\item $\D$ is strongly first order; 
\item $\D$ is logically equivalent to a finite disjunction of U-sentences; 
\item $\D(R)$ is definable in $\FO(=\!\!(\cdot), \NE, \sqcup)$. 
\end{enumerate}
\label{theo:SFO_DED_Chara}
\end{thm}
\begin{proof}$\\$
\begin{description}
\item[1. $\rightarrow$ 2.] This is Proposition \ref{propo:char_global}. 
\item[2. $\rightarrow$ 3.] 
Let $\D(R) = \bigvee_{i=1}^n \phi_i$ be a disjunction of U-sentences. 

For every $i$, let $\phi'_i(\tuple y)$ be the translation of $\phi_i$ in $\FO(=\!\!(\cdot), \NE)$ as per Proposition \ref{Usent2CNE}. Then $\D(R)$ is definable by the $\FO(=\!\!(\cdot), \NE, \sqcup)$ formula $\phi'(\tuple y) = \bigsqcup_{i =1}^n \phi'_i(\tuple y)$ : indeed, $\M \models_X \D \tuple y$ if and only if $(M, X(\tuple y)) \in \D$, that is if and only if $(M, X(\tuple y)) \models \phi_i$ for some $i = 1 \ldots n$, that is if and only if $\M \models_X \phi'_i(\tuple y)$ for some such $i$, that is if and only if $\M \models_X \phi'(\tuple y)$. 
\item[3. $\rightarrow$ 1.] Suppose that $\D$ is definable in $\FO(=\!\!(\cdot), \NE, \sqcup)$. Then, by Proposition \ref{propo:def_expr}, $\FO(\D) \leq \FO(=\!\!(\cdot), \NE, \sqcup)$. By Theorem \ref{theo:upcl_sfo}, $\FO(=\!\!(\cdot), \NE) \equiv \FO$; so, by Theorem \ref{thm:sqcup_safe}, $\FO(=\!\!(\cdot), \NE, \sqcup) \equiv \FO$, and therefore $\FO \leq \FO(\D) \leq \FO(=\!\!(\cdot), \NE, \sqcup) \equiv \FO$ and thus $\FO(\D) \equiv \FO$. 
\end{description}\qed
\end{proof}
\begin{coro}
Let $\DD$ be a family of domain-independent DEDs. Then $\FO(\DD) \equiv \FO$ if and only if every $\D \in \DD$, taken individually, is strongly first order. 
\end{coro}
\begin{proof}
If $\FO(\DD) \equiv \FO$ then every $\D \in \DD$ must be strongly first order, because $\FO \leq \FO(\D) \leq \FO(\DD) \equiv \FO$. Conversely, suppose that every $\D \in \DD$ is strongly first order, and therefore by Theorem \ref{theo:SFO_DED_Chara} is definable in $\FO(=\!\!(\cdot), \NE, \sqcup)$: then $\FO \leq \FO(\DD) \leq \FO(=\!\!(\cdot), \NE, \sqcup) \equiv \FO$, and so $\FO(\DD) \equiv \FO$.  \qed
\end{proof}
\section{Conclusions and Further Work}
In this work, we were able to characterize which dependencies are 'safe' to add to First Order Logic with Team Semantics for a very general class of dependencies that captures most dependencies of interest to Database Theory (and nearly all the dependencies that have been studied so far in Team Semantics). 

This almost entirely solves the problem of which ``reasonable'' dependencies are strongly first order; generalizing the characterization from the class DED to the class DED$^{\not =}$ would further broaden the scope of this result, but since Proposition \ref{propo:DEDhom} fails to hold for this larger class it seems that some additional ideas would be required. 

Another direction worth pursuing at this point might be to try to characterize the classes of dependencies $\DD$ for which $\FO(\DD)$ is as expressive as existential second order logic (as is the case for functional dependence atoms) or for which the model checking problem $\FO(\DD)$ is in PTIME for finite models (as is the case for inclusion atoms). This last question could also have interesting implications in descriptive complexity theory. 

\bibliographystyle{splncs04}
\bibliography{biblio}

\begin{thebibliography}{10}
\providecommand{\url}[1]{\texttt{#1}}
\providecommand{\urlprefix}{URL }
\providecommand{\doi}[1]{https://doi.org/#1}

\bibitem{abiteboul1995foundations}
Abiteboul, S., Hull, R., Vianu, V.: Foundations of Databases. Addison-Wesley (1995)

\bibitem{Deutsch2009}
Deutsch, A.: Fol modeling of integrity constraints (dependencies). In: Liu, L., {\"O}zsu, M.T. (eds.) Encyclopedia of Database Systems, pp. 1155--1161. Springer US, Boston, MA (2009). \doi{10.1007/978-0-387-39940-9_980}

\bibitem{deutsch2001optimization}
Deutsch, A., Tannen, V.: Optimization properties for classes of conjunctive regular path queries. In: Ghelli, G., Grahne, G. (eds.) Database Programming Languages. pp. 21--39. Springer Berlin Heidelberg, Berlin, Heidelberg (2002). \doi{10.1007/3-540-46093-4_2}

\bibitem{galliani12c}
Galliani, P.: The Dynamics of Imperfect Information. Ph.D. thesis, University of Amsterdam (September 2012), \url{http://dare.uva.nl/record/425951}

\bibitem{galliani12}
Galliani, P.: Inclusion and exclusion dependencies in team semantics: On some logics of imperfect information. Annals of Pure and Applied Logic  \textbf{163}(1),  68 -- 84 (2012). \doi{10.1016/j.apal.2011.08.005}

\bibitem{galliani2015upwards}
Galliani, P.: Upwards closed dependencies in team semantics. Information and Computation  \textbf{245},  124--135 (2015). \doi{10.1016/j.ic.2015.06.008}

\bibitem{galliani2016strongly}
Galliani, P.: On strongly first-order dependencies. In: Dependence Logic, pp. 53--71. Springer (2016). \doi{10.1007/978-3-319-31803-5_4}

\bibitem{galliani2019characterizing}
Galliani, P.: Characterizing downwards closed, strongly first-order, relativizable dependencies. The Journal of Symbolic Logic  \textbf{84}(3),  1136--1167 (2019). \doi{10.1017/jsl.2019.12}

\bibitem{galliani2019nonjumping}
Galliani, P.: Characterizing strongly first order dependencies: The non-jumping relativizable case. Electronic Proceedings in Theoretical Computer Science  \textbf{305},  66–82 (Sep 2019). \doi{10.4204/eptcs.305.5}

\bibitem{galliani2020safe}
Galliani, P.: Safe dependency atoms and possibility operators in team semantics. Information and Computation p. 104593 (2020)

\bibitem{galliani2022strongly}
Galliani, P.: Strongly first order, domain independent dependencies: The union-closed case. In: International Workshop on Logic, Language, Information, and Computation. pp. 263--279. Springer (2022). \doi{10.1007/978-3-031-15298-6_17}

\bibitem{galliani24}
Galliani, P.: Doubly strongly first-order dependencies. Journal of Logic and Computation p. exae056 (09 2024). \doi{10.1093/logcom/exae056}

\bibitem{gallhella13}
Galliani, P., Hella, L.: {Inclusion Logic and Fixed Point Logic}. In: Rocca, S.R.D. (ed.) Computer Science Logic 2013 (CSL 2013). Leibniz International Proceedings in Informatics (LIPIcs), vol.~23, pp. 281--295. Schloss Dagstuhl--Leibniz-Zentrum fuer Informatik, Dagstuhl, Germany (2013). \doi{10.4230/LIPIcs.CSL.2013.281}

\bibitem{gradel13}
Gr\"adel, E., V\"a\"an\"anen, J.: Dependence and independence. Studia Logica  \textbf{101}(2),  399--410 (2013). \doi{10.1007/s11225-013-9479-2}

\bibitem{hodges97}
Hodges, W.: {C}ompositional {S}emantics for a {L}anguage of {I}mperfect {I}nformation. Journal of the Interest Group in Pure and Applied Logics  \textbf{5 (4)},  539--563 (1997). \doi{10.1093/jigpal/5.4.539}

\bibitem{kanellakis1990elements}
Kanellakis, P.C.: Elements of relational database theory. In: {van Leeuwen}, J. (ed.) Formal Models and Semantics, pp. 1073--1156. Handbook of Theoretical Computer Science, Elsevier, Amsterdam (1990). \doi{10.1016/B978-0-444-88074-1.50022-6}

\bibitem{kontinen2016decidable}
Kontinen, J., Kuusisto, A., Virtema, J.: {Decidability of Predicate Logics with Team Semantics}. In: 41st International Symposium on Mathematical Foundations of Computer Science (MFCS 2016). Leibniz International Proceedings in Informatics (LIPIcs), vol.~58, pp. 60:1--60:14 (2016). \doi{10.4230/LIPIcs.MFCS.2016.60}

\bibitem{kontinennu11}
Kontinen, J., Nurmi, V.: Team logic and second-order logic. Fundamenta Informaticae  \textbf{106}(2-4),  259--272 (2011). \doi{10.3233/FI-2011-386}

\bibitem{kuusisto13}
Kuusisto, A.: Defining a double team semantics for generalized quantifiers (extended version) (2013), \url{https://trepo.tuni.fi/handle/10024/68064}, manuscript

\bibitem{vaananen07}
V\"a\"an\"anen, J.: Dependence Logic. Cambridge University Press (2007). \doi{10.1017/CBO9780511611193}

\bibitem{vaananen2023atom}
Väänänen, J.: An atom’s worth of anonymity. Logic Journal of the IGPL  \textbf{31}(6),  1078--1083 (11 2022). \doi{10.1093/jigpal/jzac074}, \url{https://doi.org/10.1093/jigpal/jzac074}

\end{thebibliography}
\end{document}